\documentclass[a4paper,11pt,leqno]{article}

\usepackage{amsfonts,latexsym,amsmath,amssymb,amsthm}
\usepackage{fullpage}
\usepackage{dsfont}
\usepackage{hyperref}
\newtheorem{theorem}{Theorem}[section]
\newtheorem{lemma}[theorem]{Lemma}
\newtheorem{proposition}[theorem]{Proposition}
\newtheorem{corollary}[theorem]{Corollary}

\newtheorem{example}[theorem]{Example}
\theoremstyle{definition}

\numberwithin{equation}{section}

\newcommand{\ls}{\leqslant}
\newcommand{\gr}{\geqslant}
\newcommand{\E}{\mathbb{E}}
\newcommand{\Pp}{\mathbb{P}}
\newcommand{\R}{\mathbb{R}}
\newcommand{\conv}{\operatorname{conv}}
\newcommand{\vrad}{{\rm vrad}}

\usepackage{setspace}
\begin{document}
\small

\title{\bf Negative moments of the support function and applications to isotropic convex bodies}
\author{Antonios Hmadi and Dimitrios-Marios Liakopoulos}
\date{}
\maketitle

\begin{abstract}\footnotesize
We study lower tails of support functions of isotropic convex bodies.
An endpoint negative moment estimate, combined with a spherical cap argument, gives polynomial deviation estimates for the inradius of random projections in every dimension; in the unconditional class this yields the optimal order of the median inradius and shows that the cube is extremal.
The same argument gives estimates for convex hulls of independent rotations and, after an additional projection, a mixed rotation--projection theorem.
We also prove weighted lower and upper estimates for Minkowski averages of rotated polars and their random projections, as well as an inradius dependent estimate for the geometric distance of global averages from the Euclidean ball.
A family of isotropic product cylinders shows that the latter estimate is sharp, up to absolute constants, throughout the possible range of the inradius.
For the cube, an explicit negative moment computation gives a quantitative projected $(m,k,n)$-profile and, in full dimension, recovers the classical order $n/\ln(en)$ for bounded geometric distance.
Finally, we determine the sharp volume profile of intersections of independent rotations in the unconditional isotropic class and show that the cube is extremal.
\end{abstract}

\section{Introduction}\label{section:1}

The geometry of an isotropic convex body is often studied through the distribution of its norm, support function and radial function in random directions.
Random projections and random rotations provide two natural ways of revealing Euclidean structure, while negative moments measure the lower tail which is not seen by the usual concentration inequalities.
The purpose of this paper is to combine these points of view.
Our main results concern small values of the support function, inradii of random projections, intersections of independently rotated isotropic convex bodies, Minkowski averages and the geometric distance of global averages from the Euclidean ball in the isotropic position.

Let $K\subseteq\R^n$ be a convex body with the origin in its interior.
We denote by $h_K$ and $p_K$ its support function and Minkowski functional.
For $q\neq0$ we set
$$ M_q(K)=\left(\int_{S^{n-1}}p_K(\theta)^q\,d\sigma(\theta)\right)^{1/q},\qquad w_q(K)=\left(\int_{S^{n-1}}h_K(\theta)^q\,d\sigma(\theta)\right)^{1/q}, $$
and write $M(K)=M_1(K)$, $w(K)=w_1(K)$ and $b(K)=\max_{S^{n-1}}p_K$.

Polar integration gives
$$ w_{-n}(K)^{-n}=\frac{|K^\circ|}{\omega_n}, \qquad \vrad(K)^n=\frac{|K|}{\omega_n}, $$
where $\vrad(K)$ is the volume radius of $K$, and therefore
$$ \left(\frac{w_{-n}(K)}{\vrad(K)}\right)^n =\frac{\omega_n^2}{|K|\,|K^\circ|}. $$
If $K$ is centered, the Blaschke--Santal\'o inequality~\cite{BGVV-book} gives
\begin{equation}\label{eq:intro-santalo-negative-support}
 w_{-n}(K)\gr\vrad(K),\qquad
 \sigma\{\theta:h_K(\theta)\ls t\,\vrad(K)\}\ls\min\{1,t^n\},\quad t>0.
\end{equation}
Equality in the first inequality holds exactly for ellipsoids centered at the origin.
In particular, if $|K|=1$, then
\begin{equation}\label{eq:intro-small-support}
 \sigma\{\theta:h_K(\theta)\ls t\sqrt n\}\ls\min\{1,(Ct)^n\},\qquad t>0.
\end{equation}
This is the endpoint power mean form of the volume product inequality; related formulations appear in Santal\'o~\cite{Santalo-1949} and Lutwak~\cite{Lutwak-1983}.

For $0<\vartheta<1$ put
$$ d_\ast(K,\vartheta)=\min\left\{-\ln\sigma\{h_K\ls\vartheta w(K)\},n\right\}. $$
The parameter $d_\ast(K,1/2)$ was introduced by Klartag and Vershynin~\cite{Klartag-Vershynin-2007}.
The endpoint estimate has the following consequence.

\begin{theorem}\label{th:dstar}
Let $K\subseteq\R^n$ be a centered convex body of volume one.
Then, for every $0<\vartheta<1$,
$$ d_\ast(K,\vartheta)\gr \min\left\{n,n\max\left\{0,\ln\left(\frac{c_1\sqrt n}{\vartheta w(K)}\right)\right\}\right\}. $$
In particular, $d_\ast(K,\vartheta)=n$ whenever
$$ \vartheta\ls c_2\frac{\sqrt n}{w(K)}. $$
\end{theorem}

For isotropic convex bodies, the optimal mean-width estimate of Bizeul~\cite{Bizeul-mmstar-2026}, after rescaling from covariance-one normalization, gives $w(K)\ls CL_K\sqrt{n\ln(en)}$. Together with the slicing theorem of Klartag and Lehec~\cite{KL}, this implies $w(K)\ls C\sqrt{n\ln(en)}$, and hence $d_\ast(K,\vartheta)=n$ whenever $\vartheta\ls c/\sqrt{\ln(en)}$.

\smallskip 

Our first projection result follows from an estimate of Klartag and Vershynin~\cite{Klartag-Vershynin-2007} relating negative moments of projection inradii to negative moments of the support function, together with the estimate of Pajor and Tomczak-Jaegermann~\cite{Pajor-Tomczak-Jaegermann-1986} for the low-$M^\ast$ theorem.

\begin{theorem}\label{th:inradius}
There exist absolute constants $c,c'>0$ such that the following holds.
Let $K\subseteq\R^n$ be a symmetric convex body of volume one and let $1\ls k<(n-1)/4$.
Then a Haar-distributed $F\in G_{n,k}$ satisfies
$$ r(P_FK)\gr c\max\left\{\frac n{w(K)},\frac1{M(K)}\right\} $$
with probability at least $1-e^{-c'k}$, where $r(P_FK)$ denotes the inradius of $P_FK$.
\end{theorem}

In particular, if $K$ is symmetric and isotropic, then \eqref{eq:isotropic-spherical-means} gives
$$ r(P_FK)\gr c\sqrt{\frac n{\ln(en)}} $$
with probability at least $1-e^{-c'k}$ in the range of Theorem~\ref{th:inradius}.

We also obtain two all dimensional projection results.
Random sections of isotropic convex bodies were studied in~\cite{Giannopoulos-Hioni-Tsolomitis-2016}; high-probability regular positions and random Gelfand numbers were developed in~\cite{Milman-Yifrach-2021}.
Let $K\subseteq\R^n$ be isotropic, let $1\ls k<n$, let $F$ be Haar-distributed on $G_{n,k}$ and put $ \Delta=n-k+1. $

\begin{theorem}\label{th:projection-tail}
There exist absolute constants $c,C>0$ such that, for every $\tau\gr1$,
\begin{equation}\label{eq:projection-tail-all-dimensional}
 \Pp\left\{r(P_FK)<\frac{c\sqrt n}{\tau}\exp\left(-\frac{Ck}{\Delta}\right)\right\} \ls\tau^{-\Delta}.
\end{equation}
In particular, writing $d=n-k$, for every $0<\delta<1$ there is $c_\delta>0$ such that, whenever $d\gr\delta n$,
$$ \Pp\left\{P_FK\not\supseteq\frac{c_\delta\sqrt n}{\tau}B_F\right\} \ls\tau^{-(d+1)}. $$
For $d\gr n/2$, the constant $c_\delta$ may be chosen absolute.
\end{theorem}

A second argument, based on Gordon's theorem~\cite{Gordon-1988}, gives, with probability at least $1-Ce^{-cd}$, the complementary bound
$$ r(P_FK)\gr c\sqrt{\frac d{\ln(en)}}, \qquad d=n-k. $$
In the unconditional class the optimal order of the median inradius can be determined.
Let $\mathcal I_n^{\rm unc}$ be the class of volume-one unconditional isotropic bodies and, for $1\ls d<n$ and Haar-distributed $F\in G_{n,n-d}$, set
$$ \mathfrak r_{n,d}^{\rm unc} =\sup\left\{s>0:\inf_{K\in\mathcal I_n^{\rm unc}} \Pp\{r(P_FK)\gr s\}\gr\frac12\right\}. $$

\begin{theorem}\label{th:unconditional-extremal-profile}
For every $1\ls d<n$,
\begin{equation}\label{eq:unconditional-extremal-profile-estimate}
 \mathfrak r_{n,d}^{\rm unc} \simeq\max\left\{1,\sqrt{\frac d{\ln(en/d)}}\right\}.
\end{equation}
In particular, the volume-one cube is extremal, up to absolute constants, among unconditional isotropic bodies in every projection dimension.
\end{theorem}

We next turn to independent random rotations.
The comparison between sections and intersections of a few rotations originates in~\cite{Milman-1991,Giannopoulos-Milman-1997}; quantitative local-to-global forms were obtained in~\cite{Vershynin-2006}.
The theorem below follows from a mixed rotation--projection statement proved in Section~\ref{section:4}.

\begin{theorem}\label{th:m-rotation-main}
There exist absolute constants $c,C>0$ such that the following holds.
Let $n\gr2$ and $m\gr2$, let $K_1,\ldots,K_m\subseteq\R^n$ be isotropic convex bodies and let $U_1=I_n,U_2,\ldots,U_m$ be independent Haar-distributed orthogonal transformations.
Put $ q_m=(m-1)n+1. $
Then, for every $\tau\gr1$,
$$ \Pp\left\{R\left(\bigcap_{j=1}^mU_jK_j^\circ\right)>\frac{C\tau}{\sqrt n}\right\} \ls\tau^{-q_m}. $$
Equivalently, with probability at least $1-\tau^{-q_m}$,
\begin{equation}\label{eq:m-rotation-hull}
 \conv\{U_1K_1,\ldots,U_mK_m\} \supseteq\frac{c\sqrt n}{\tau}B_2^n.
\end{equation}
\end{theorem}

For every fixed $\tau_0>1$ the exceptional probability is exponentially small in $(m-1)n+1$.
The factor $\sqrt n$ in this inclusion is optimal: for the volume-one cube $Q_n=[-1/2,1/2]^n$,
$$ \conv\{U_1Q_n,\ldots,U_mQ_n\}\subseteq\frac{\sqrt n}{2}B_2^n $$
for every choice of the rotations.
More generally, if an independent $F\in G_{n,k}$ is introduced, the tail exponent becomes $mn-k+1$.

The next result is obtained as the equal body consequence of a weighted estimate for heterogeneous bodies.

\begin{theorem}\label{th:minkowski-average-negative-moments}
Let $K\subseteq\R^n$ be a symmetric convex body and let $U_1,\ldots,U_m$ be independent Haar-distributed orthogonal transformations.
Put
$$ A_m=\frac1m\sum_{j=1}^mU_j^*(K^\circ),\qquad M=M(K),\qquad b=b(K). $$
Assume that, for some $p_0>0$,
$$ M_{-p}(K)\gr c_0M\qquad(0<p\ls p_0). $$
Let $2n\ls q\ls mp_0$ and put $\Delta=q-n+1$.
Then, for every $u\gr0$, with probability at least $1-e^{-u}$,
\begin{equation}\label{eq:minkowski-average-lower-inclusion}
 A_m\supseteq cM\left(\frac{cM}{b}\right)^{\frac{n-1}{\Delta}}e^{-u/\Delta}B_2^n.
\end{equation}
Moreover, with probability at least $1-e^{-cn}$,
\begin{equation}\label{eq:minkowski-average-upper-inclusion}
 A_m\subseteq C\left(M+\frac b{\sqrt m}\right)B_2^n.
\end{equation}
Consequently, with probability at least $1-e^{-u}-e^{-cn}$,
\begin{equation}\label{eq:minkowski-average-geometric-distance}
 d_{\rm G}(A_m,B_2^n) \ls C\left(1+\frac b{M\sqrt m}\right) \left(\frac{Cb}{M}\right)^{\frac{n-1}{\Delta}}e^{u/\Delta}.
\end{equation}
\end{theorem}

The same method also applies after an independent random projection.
In Section~\ref{section:5} we prove a projected weighted estimate in which the residual exponent is $q-k+1$ and the upper deviation term is reduced by the factor $\sqrt{k/n}$.
In particular, if $F$ is an independent Haar-distributed element of $G_{n,k}$ and
$$ A_m=\frac1m\sum_{j=1}^mU_j^*(Q_n^\circ), \qquad L_{m,k}=\ln\left(e+\min\left\{n,\frac{mn}{k\ln(en)}\right\}\right), $$
then, with probability at least $1-Ce^{-ck}$,
$$ d_{\rm G}(P_FA_m,B_F) \ls C\left(\sqrt{\frac{\ln(en)}{L_{m,k}}}+\sqrt{\frac{k}{mL_{m,k}}}\right). $$

Next, we consider the global averaging problem. 
Bourgain, Lindenstrauss and V.~Milman~\cite{BLM} proved that sufficiently many random rotations of a symmetric convex body give an almost Euclidean average, while Chasapis and Giannopoulos~\cite{Chasapis-Giannopoulos-2016} obtained an isomorphic version for fixed number of rotations of bodies in John's position. 
Litvak, V.~D.~Milman and Schechtman~\cite{Litvak-Milman-Schechtman-1997,Litvak-Milman-Schechtman-1998} determined, up to absolute constants, the number of rotations needed for a uniformly Euclidean average in terms of $(b(K)/M(K))^2$.
Combining the Bourgain--Lindenstrauss--Milman theorem with the Euclidean regularization scheme of Fresen~\cite{Fresen-2015} and Chasapis--Giannopoulos~\cite{Chasapis-Giannopoulos-2016}, adapted to the isotropic position, gives the following inradius dependent form.

\begin{theorem}\label{th:global-dvoretzky}
There exist an absolute integer $k_0\gr2$ and absolute constants $c,C>0$ such that the following holds.
Let $K\subseteq\R^n$ be a symmetric isotropic convex body.
For every integer $k\gr k_0$, independent Haar-distributed orthogonal transformations $U_1,\ldots,U_k\in O(n)$ satisfy
\begin{equation}\label{eq:isotropic-global}
 d_{\rm G}\left(\frac1k\sum_{i=1}^kU_i^*(K^\circ),B_2^n\right) \ls C\max\left\{1,\frac1{r(K)}\sqrt{\frac nk}\right\}
\end{equation}
with probability at least $1-e^{-cn}$.
\end{theorem}

Since $K$ has volume one, $r(K)\ls\omega_n^{-1/n}\simeq\sqrt n$, while every isotropic body satisfies $r(K)\gr c$.
Thus \eqref{eq:isotropic-global} gives bounded geometric distance as soon as $k\gr Cn/r(K)^2$.
In particular, the number of rotations ranges from order $n$ when $r(K)\simeq1$ to an absolute number when $r(K)\simeq\sqrt n$.
Proposition~\ref{prop:product-cylinder-global-sharpness} shows that this dependence on the inradius is optimal, up to absolute constants, throughout this range.

Finally, we consider the volume of intersections of independent rotations.
General estimates for the volume and radius of $C\cap U(C)$ were obtained by Brazitikos and Stavrakakis~\cite{Brazitikos-Stavrakakis-2014}.
In the unconditional isotropic class we determine the correct order for every number of rotations and show that the cube is extremal.
Put
$$ Q_n=[-1/2,1/2]^n,\qquad \Lambda_{m,n}=\min\{n,\ln(e+m)\}. $$
The lower estimate below is deterministic, while the matching upper estimate for the cube holds with exponentially high probability.

\begin{theorem}\label{th:unconditional-random-intersection-volume}
There exist absolute constants $c,C,c'>0$ with the following property.
Let $n,m\gr2$, let $K_1,\ldots,K_m\subseteq\R^n$ be unconditional isotropic convex bodies, and let $U_1,\ldots,U_m\in O(n)$.
Then
\begin{equation}\label{eq:unconditional-random-intersection-lower}
 \left|\bigcap_{i=1}^mU_iK_i\right|^{1/n}\gr\frac{c}{\sqrt{\Lambda_{m,n}}}.
\end{equation}
On the other hand, if $U_1=I_n$ and $U_2,\ldots,U_m$ are independent Haar-distributed orthogonal transformations, then
\begin{equation}\label{eq:cube-random-intersection-mean-volume}
 \E\left|\bigcap_{i=1}^mU_iQ_n\right|\ls\left(\frac{C}{\sqrt{\Lambda_{m,n}}}\right)^n.
\end{equation}
Consequently,
\begin{equation}\label{eq:cube-random-intersection-probability}
 \Pp\left\{\frac{c}{\sqrt{\Lambda_{m,n}}}\ls\left|\bigcap_{i=1}^mU_iQ_n\right|^{1/n}\ls\frac{C}{\sqrt{\Lambda_{m,n}}}\right\}\gr1-e^{-c'n}.
\end{equation}
In particular, the volume-one cube is extremal, up to absolute constants, for the median volume radius of intersections of $m$ independent rotations in the unconditional isotropic class.
\end{theorem}

The paper is organized as follows.
Section~\ref{section:2} recalls notation and preliminary facts.
Section~\ref{section:3} contains the endpoint estimates and their applications to projections.
Section~\ref{section:4} treats independent rotations and the mixed rotation--projection theorem.
Section~\ref{section:5} studies small deviations, weighted Minkowski averages and their random projections.
Section~\ref{section:6} proves Theorem~\ref{th:global-dvoretzky}, its application to the cube and the sharpness of its dependence on the inradius.
Section~\ref{section:7} determines the sharp volume profile of random intersections in the unconditional isotropic class.
The calculations for the cube and the estimates for the weak projection profile introduced in Section~\ref{section:3} are collected in Appendix~\ref{appendix:special-calculations}.

\section{Notation and preliminary facts}\label{section:2}

We work in $\R^n$ with its standard Euclidean structure.
The Euclidean norm, unit ball and unit sphere are denoted by $|\cdot|$, $B_2^n$ and $S^{n-1}$, respectively.
Lebesgue measure in $\R^n$ is also denoted by $|\cdot|$, and $\omega_n=|B_2^n|$.
For $1\ls p\ls\infty$, $B_p^n$ denotes the unit ball of $\ell_p^n$; in particular, $B_\infty^n=[-1,1]^n$.
We write $\sigma$ for the rotationally invariant probability measure on $S^{n-1}$ and $\nu$ for Haar probability measure on the orthogonal group $O(n)$.
For $1\ls k\ls n$, $G_{n,k}$ denotes the set of $k$-dimensional linear subspaces of $\R^n$, and $\nu_{n,k}$ denotes its rotationally invariant probability measure.
If $E\subseteq\R^n$ is a nonzero linear subspace, we write $B_E=B_2^n\cap E$, $S_E=S^{n-1}\cap E$ and $\sigma_E$ for the rotationally invariant probability measure on $S_E$.
If $F\in G_{n,k}$, then $P_F$ is the orthogonal projection onto $F$.
For $I\subseteq\{1,\ldots,n\}$, let $E_I=\operatorname{span}\{e_i:i\in I\}$, let $P_I=P_{E_I}$ and write $B_2^I=B_2^n\cap E_I$, where $e_1,\ldots,e_n$ are the standard coordinate vectors.
We write $I_n$ for the identity operator on $\R^n$ and $U^*$ for the adjoint of $U$; in particular, $U^*=U^{-1}$ for $U\in O(n)$.
The symbols $\Pp$ and $\E$ denote probability and expectation with respect to all random objects under consideration.
Throughout, $G$ denotes a standard Gaussian vector in $\R^n$.
The symbols $c,C,c_1,C_1,\ldots$ denote positive absolute constants whose values may change from line to line; $a\simeq b$ means that $c_1a\ls b\ls C_1a$ for absolute constants $c_1,C_1>0$, and a subscript indicates the parameters on which these constants may depend.
For a set $A$, $\mathds{1}_A$ denotes its indicator function.
For $x\in\R^n$ and a nonempty set $A\subseteq\R^n$ put $\operatorname{dist}(x,A)=\inf\{|x-y|:y\in A\}$.

A convex body is a compact convex set with nonempty interior.
It is symmetric if $K=-K$ and centered if its barycenter is the origin.
It is unconditional if $(x_1,\ldots,x_n)\in K$ implies $(\varepsilon_1x_1,\ldots,\varepsilon_nx_n)\in K$ for every choice of $\varepsilon_i\in\{-1,1\}$.
For a body $K$ containing the origin in its interior, its radial function, support function and Minkowski functional are defined by
$$ \rho_K(\theta)=\sup\{t\gr0:t\theta\in K\},\qquad h_K(x)=\sup_{y\in K}\langle x,y\rangle,\qquad p_K(x)=\inf\{t>0:x\in tK\}. $$
Its polar body is
$$ K^\circ=\{x\in\R^n:\langle x,y\rangle\ls1\ \text{for all }y\in K\}. $$
If $K$ is symmetric, then $p_K=\|\cdot\|_K$.
On the sphere, $p_K=\rho_K^{-1}$, $\rho_{K^\circ}=h_K^{-1}$ and $h_{K^\circ}=p_K$.
We shall also use without further comment that $(UK)^\circ=U(K^\circ)$, that $(P_FK)^\circ_F=K^\circ\cap F$ (the polar being taken in $F$), and that
$$ \left(\bigcap_{j=1}^mK_j\right)^\circ=\conv\left(\bigcup_{j=1}^mK_j^\circ\right). $$
Moreover, $h_{K+L}=h_K+h_L$, $h_{\conv(K\cup L)}=\max\{h_K,h_L\}$ and $h_{U^*(K^\circ)}(x)=p_K(Ux)$ for $U\in O(n)$.

We write
$$ R(K)=\max_{x\in K}|x|,\qquad r(K)=\max\{r>0:rB_2^n\subseteq K\},\qquad \vrad(K)=\left(\frac{|K|}{\omega_n}\right)^{1/n}. $$
Thus $r(K)=1/R(K^\circ)$ and $r(K)=\min_{S^{n-1}}h_K$, while $b(K)=\max_{S^{n-1}}p_K=1/r(K)$.
Polar integration gives
\begin{equation}\label{eq:polar-integration}
 |K|=\omega_n\int_{S^{n-1}}\rho_K^n\,d\sigma=\omega_n\int_{S^{n-1}}p_K^{-n}\,d\sigma, \qquad |K^\circ|=\omega_n\int_{S^{n-1}}h_K^{-n}\,d\sigma,
\end{equation}
and $\omega_n^{1/n}\simeq n^{-1/2}$.
We use the notation $M_q(K)$, $w_q(K)$, $M(K)$, $w(K)$, $b(K)$ and $d_\ast(K,\vartheta)$ introduced in the Introduction.
Polarity gives $w_q(K)=M_q(K^\circ)$ and both $M_q(K)$ and $w_q(K)$ are nondecreasing functions of $q$.
For a symmetric body $K$, the function $p_K$ is $b(K)$-Lipschitz on $S^{n-1}$.

For a non-negative random variable $Y$ and $q>0$ we write
$$ \|Y\|_{L_{q,\infty}}=\sup_{s>0}s\,\Pp\{Y>s\}^{1/q}. $$
This is the usual weak $L_q$, or Lorentz $L_{q,\infty}$, quasi-norm. 
For $A>0$, the inequality $\|Y\|_{L_{q,\infty}}\ls A$ is equivalent to
$$ \Pp\{Y>A\tau\}\ls\tau^{-q}\qquad(\tau>0). $$
For a symmetric convex body $C\subseteq\R^n$ and $0\ls d<n$ we use the convention
$$ c_d(C,\ell_2^n)=\inf\{R(C\cap F):F\in G_{n,n-d}\} $$
for its $d$-th Gelfand width in $\ell_2^n$. 
This is equivalent to the convention in~\cite{Foucart-Pajor-Rauhut-Ullrich-2010}, where the infimum is taken over the kernels of $d\times n$ matrices.
For symmetric bodies
$$ d_{\rm G}(K,L)=\inf\{ab:a^{-1}L\subseteq K\subseteq bL,\ a,b>0\} $$
is the geometric distance of $K$ and $L$.

A convex body $K$ is isotropic if $|K|=1$, its barycenter is the origin and
$$ \int_K\langle x,\theta\rangle^2\,dx=L_K^2\qquad(\theta\in S^{n-1}) $$
for some $L_K>0$.
Every convex body has an isotropic affine image, unique up to orthogonal transformations; see \cite{GPT-survey-2025}.
The slicing problem was posed by Bourgain~\cite{Bourgain-1986}.
We shall use its solution by Klartag and Lehec~\cite{KL}, which gives
\begin{equation}\label{eq:universal-slicing}
 c\ls L_K\ls L_0
\end{equation}
for an absolute constant $L_0$.
We also use the optimal mean-width estimate of Bizeul~\cite{Bizeul-mmstar-2026}, which improves the earlier bound of E.~Milman~\cite{EMilman-2014}, and the gauge estimate of Bizeul and Klartag~\cite{Bizeul-Klartag-2025}. After rescaling from covariance-one normalization, and using the dimension-free bound for the third-moment parameter of Letwin~\cite{Letwin-2026} as recorded in~\cite{Bizeul-mmstar-2026}, these give
\begin{equation}\label{eq:isotropic-spherical-means}
 w(K)\ls CL_K\sqrt{n\ln(en)}, \qquad M(K)\ls C\sqrt{\frac{\ln(en)}n}.
\end{equation}
We refer to \cite{BGVV-book} for background on isotropic convex bodies and log-concave measures and to \cite{AGA-book,AGA-book-2} for the local theory of finite-dimensional normed spaces.

\section{Small support values and random projection inradii}\label{section:3}

We first justify the endpoint estimate used in the Introduction.
By \eqref{eq:polar-integration} and the Blaschke--Santal\'o inequality~\cite{BGVV-book}, every centered convex body $C$ satisfies
$$ |C^\circ|=\omega_n\int_{S^{n-1}}h_C(\theta)^{-n}\,d\sigma(\theta)\ls\frac{\omega_n^2}{|C|}. $$
This is equivalent to the first inequality in \eqref{eq:intro-santalo-negative-support}.
Markov's inequality yields the tail estimate
$$ \sigma\{h_C\ls t\,\vrad(C)\} \ls t^n\vrad(C)^n\int_{S^{n-1}}h_C^{-n}\,d\sigma \ls t^n. $$
Finally, if $|C|=1$, then $\vrad(C)=\omega_n^{-1/n}\simeq\sqrt n$, and the volume-one form \eqref{eq:intro-small-support} follows as well.

\begin{proof}[Proof of Theorem~$\ref{th:dstar}$]
Apply \eqref{eq:intro-small-support} with $t=\vartheta w(K)/\sqrt n$ and take minus logarithms.
The last assertion follows by decreasing the absolute constant $c>0$ so that the resulting probability is at most $e^{-n}$.
\end{proof}

For the proof of Theorem~\ref{th:inradius} we recall two standard estimates.
Klartag and Vershynin~\cite{Klartag-Vershynin-2007} proved that, if $K$ is symmetric, $1\ls q<n$ and $1\ls k<q/4$, then
$$ \left(\int_{G_{n,k}}r(P_FK)^{-k}\,d\nu_{n,k}(F)\right)^{1/k}\ls C\frac{w(K)}{w_{-q}(K)^2}. $$
We shall also use the low $M^\ast$ estimate in the following form (see, for example, \cite{AGA-book}): for a symmetric convex body $C$ and $1\ls k<n$, a Haar-distributed $F\in G_{n,k}$ satisfies
$$ R(C\cap F)\ls C\sqrt{\frac n{n-k}}\,w(C) $$
with probability at least $1-\exp(-c(n-k))$.

\begin{proof}[Proof of Theorem~$\ref{th:inradius}$]
Apply the first estimate with $q=n-1$.
By \eqref{eq:intro-santalo-negative-support} and monotonicity of the spherical moments,
$$ w_{-(n-1)}(K)\gr w_{-n}(K)\gr\omega_n^{-1/n}\gr c\sqrt n. $$
Hence
$$ \left(\int_{G_{n,k}}r(P_FK)^{-k}\,d\nu_{n,k}(F)\right)^{-1/k}\gr c\frac n{w(K)}. $$
Markov's inequality gives $r(P_FK)\gr cn/w(K)$ with probability at least $1-e^{-k}$.

Applying the low $M^\ast$ estimate to $C=K^\circ$ and using $w(K^\circ)=M(K)$, we get
$$ R(K^\circ\cap F)\ls C\sqrt{\frac n{n-k}}\,M(K) $$
outside a set of probability at most $e^{-c(n-k)}$.
Since $(P_FK)^\circ_F=K^\circ\cap F$,
$$ P_FK\supseteq c\sqrt{1-\frac kn}\,\frac1{M(K)}B_F. $$
Since $k<(n-1)/4$, the factor $\sqrt{1-k/n}$ is bounded below by an absolute constant.
Intersecting the two good events completes the proof.
\end{proof}

Let $X$ be uniformly distributed on an isotropic convex body $K$.
Paouris' deviation inequality~\cite{Paouris-2006} gives
\begin{equation}\label{eq:paouris-background}
 \Pp\{|X|>Ct\sqrt nL_K\}\ls e^{-ct\sqrt n}, \qquad t\gr1.
\end{equation}

\begin{lemma}\label{lem:recentered-core}
There exist absolute constants $c,C>0$ such that every isotropic convex body $K\subseteq\R^n$, $n\gr2$, admits a centered convex body $\widehat K$ of volume one satisfying
\begin{equation}
 R(\widehat K)\ls C\sqrt n,
\end{equation}
\begin{equation}
 \widehat K\subseteq\lambda_nK, \qquad 1\ls\lambda_n\ls1+Ce^{-c\sqrt n},
\end{equation}
and
\begin{equation}\label{eq:recentered-core-polar-comparison}
 K^\circ\subseteq\lambda_n\widehat K^\circ, \qquad \widehat K^\circ\supseteq\frac c{\sqrt n}B_2^n.
\end{equation}
Moreover, for every $t>0$,
\begin{equation}\label{eq:recentered-core-tail}
 \sigma\{\rho_{\widehat K^\circ}\gr t\} \ls\min\left\{1,\left(\frac C{t\sqrt n}\right)^n\right\}.
\end{equation}
\end{lemma}

\begin{proof}
Choose a sufficiently large absolute $A$ and put
$$ K_0=K\cap A\sqrt nL_KB_2^n, \qquad v=|K_0|, \qquad R_0=A\sqrt nL_K. $$
By \eqref{eq:paouris-background}, $1-v\ls e^{-c\sqrt n}$, and in particular $v\gr1/2$.
If $z$ is the barycenter of $K_0$, the fact that $K$ is centered gives
$$ z=\frac1v\int_{K_0}x\,dx=-\frac1v\int_{K\setminus K_0}x\,dx. $$
Consequently,
$$ |z|\ls\frac1v\E\left[|X|\mathds{1}_{\{|X|>R_0\}}\right]. $$
Integration of the tail gives
$$ \E\left[|X|\mathds{1}_{\{|X|>R_0\}}\right]=R_0\Pp\{|X|>R_0\}+\int_{R_0}^\infty\Pp\{|X|>s\}\,ds\ls CL_Ke^{-c\sqrt n}, $$
and therefore $|z|\ls CL_Ke^{-c\sqrt n}$.
On the other hand, for every $\theta\in S^{n-1}$, the one-dimensional reverse H\"older inequality for log-concave random variables (see, for example, \cite{BGVV-book}) applied to $\langle X,\theta\rangle$ gives
$$ h_K(\theta)\gr\E\langle X,\theta\rangle_+ =\frac12\E|\langle X,\theta\rangle|\gr cL_K. $$
Here the equality follows from $\E\langle X,\theta\rangle=0$ and the identity $y_+=(|y|+y)/2$.
Thus $r(K)\gr cL_K$ and $-z\in Ce^{-c\sqrt n}K$.
Therefore
$$ \widehat K=v^{-1/n}(K_0-z) \subseteq(1+Ce^{-c\sqrt n})K. $$
The body $\widehat K$ is centered, has volume one and satisfies $R(\widehat K)\ls C\sqrt nL_K\ls C\sqrt n$ by \eqref{eq:universal-slicing}.
Polarity gives \eqref{eq:recentered-core-polar-comparison}, while centered Santal\'o and polar integration yield \eqref{eq:recentered-core-tail}.
Consequently, every monotone positively homogeneous operation applied to $K^\circ$ is controlled, up to the factor $\lambda_n$, by the same operation applied to $\widehat K^\circ$.
\end{proof}

The following lemma is the spherical cap step underlying the low-$M^\ast$ and escape through a mesh methods of Pajor and Tomczak-Jaegermann~\cite{Pajor-Tomczak-Jaegermann-1986} and Gordon~\cite{Gordon-1988}.

\begin{lemma}\label{lem:cap-circumradius}
Let $D$ be a convex body in a $k$-dimensional Euclidean space $E$, $k\gr1$, and assume that $rB_E\subseteq D$.
If $R\gr r$ and $Ru\in D$ for some $u\in S_E$, then
$$ \rho_D(\theta)\gr\frac R2\qquad\text{whenever}\qquad |\theta-u|\ls\frac rR. $$
Moreover,
$$ \sigma_E\left\{\theta:\rho_D(\theta)\gr\frac R2\right\} \gr\frac12\left(\frac r{8R}\right)^{k-1}. $$
\end{lemma}

\begin{proof}
If $|\theta-u|\ls r/R$, then $R(\theta-u)\in rB_E\subseteq D$.
Since $Ru\in D$ and $D$ is convex,
$$ \frac R2\theta=\frac12Ru+\frac12R(\theta-u)\in D. $$
For $k\gr2$, the usual estimate for the measure of a spherical cap, applied with Euclidean radius $r/(2R)$, gives
$$ \sigma_E\left\{\theta:|\theta-u|\ls\frac rR\right\}\gr\frac12\left(\frac r{8R}\right)^{k-1}; $$
see, for example, \cite[pp.~10--12]{Ball-1997}.
For $k=1$, the sphere consists of two points of mass $1/2$, and one of them is the extremal direction $u$.
\end{proof}

\begin{theorem}\label{th:master-amplification}
Let $(C,F)$ be a random pair such that $C$ is a convex body in $\R^n$ and $F\in G_{n,k}$, and assume that
$$ rB_2^n\subseteq C $$
for every realization.
Put $d=k-1$ and
$$ \overline\alpha(t)=\E\,\big(\sigma_F\{\theta\in S_F:\rho_C(\theta)\gr t\}\big). $$
Then, for every $s>0$,
\begin{equation}\label{eq:master-amplification}
 \Pp\{R(C\cap F)>s\} \ls2\sum_{\ell=0}^\infty \left(\frac{2^{\ell+4}s}{r}\right)^d \overline\alpha(2^{\ell-1}s).
\end{equation}
\end{theorem}

\begin{proof}
For $\ell\gr0$ put
$$ \mathcal E_\ell=\{2^\ell s<R(C\cap F)\ls2^{\ell+1}s\}. $$
On $\mathcal E_\ell$, Lemma~\ref{lem:cap-circumradius} gives
$$ \sigma_F\{\theta\in S_F:\rho_C(\theta)\gr2^{\ell-1}s\} \gr\frac12\left(\frac r{2^{\ell+4}s}\right)^d. $$
Markov's inequality therefore yields
$$ \Pp(\mathcal E_\ell) \ls2\left(\frac{2^{\ell+4}s}{r}\right)^d \overline\alpha(2^{\ell-1}s). $$
Summing over $\ell\gr0$ proves \eqref{eq:master-amplification}.
\end{proof}

\begin{corollary}\label{cor:master-polynomial}
Assume in addition that, for some $a,\beta>0$,
$$ \overline\alpha(t)\ls\min\left\{1,\left(\frac at\right)^\beta\right\},\qquad t>0. $$
If $\Delta=\beta-d>0$, then
\begin{equation}\label{eq:master-polynomial-exact}
 \Pp\{R(C\cap F)>s\} \ls\frac{2^{4d+\beta+1}}{1-2^{-\Delta}} \left(\frac ar\right)^d \left(\frac as\right)^\Delta.
\end{equation}
If moreover $\beta\gr2d$ and $\Delta\gr1$, then, for every $\tau\gr1$,
\begin{equation}\label{eq:master-polynomial-absolute}
 \Pp\left\{R(C\cap F)>Ca\left(\frac ar\right)^{d/\Delta}\tau\right\} \ls\tau^{-\Delta},
\end{equation}
where $C>0$ is absolute.
\end{corollary}

\begin{proof}
Insert the assumed bound into \eqref{eq:master-amplification}.
The $\ell$-th term is at most
$$ 2^{4d+\beta+1} \left(\frac ar\right)^d \left(\frac as\right)^\Delta2^{-\ell\Delta}. $$
Summation gives \eqref{eq:master-polynomial-exact}.
If $\beta\gr2d$, then $d/\Delta\ls1$ and $\beta/\Delta\ls2$, so all numerical factors in \eqref{eq:master-polynomial-exact} are absorbed by an absolute dilation of $s$.
\end{proof}

The use of negative moments to pass from one-dimensional small-ball information to diameters of random sections goes back to Klartag and Vershynin~\cite{Klartag-Vershynin-2007}; the formulation above records explicitly the residual exponent $\Delta=\beta-k+1$ for a random pair $(C,F)$.

The loss of $k-1$ is unavoidable: for $D=\conv(rB_E,Re_1)$, the directions on which $\rho_D$ is comparable to $R$ form a cap of measure comparable to $(r/R)^{k-1}$.

We now apply Theorem~\ref{th:master-amplification} to random sections.
Let $C\subseteq\R^n$ be a convex body and $F\in G_{n,k}$.
For $t>0$ put
$$ Z_t(F)=\sigma_F\{\theta\in S_F:\rho_C(\theta)\gr t\}. $$
The invariant probability measure on the incidence space gives
\begin{equation}\label{eq:grassmann-incidence}
 \int_{G_{n,k}}Z_t(F)\,d\nu_{n,k}(F) =\sigma\{\theta\in S^{n-1}:\rho_C(\theta)\gr t\}.
\end{equation}

\begin{proof}[Proof of Theorem~$\ref{th:projection-tail}$]
Let $\widehat K$ be the body from Lemma~\ref{lem:recentered-core} and set $C=\widehat K^\circ$.
Then
$$ C\supseteq\frac c{\sqrt n}B_2^n $$
and
$$ \sigma\{\rho_C\gr t\}\ls\min\left\{1,\left(\frac C{t\sqrt n}\right)^n\right\}. $$
Apply \eqref{eq:grassmann-incidence} and Corollary~\ref{cor:master-polynomial} to the pair $(C,F)$ with $ a\simeq r\simeq n^{-1/2}$, $\beta=n$ and $d=k-1.$
Then $\Delta=\beta-d=n-k+1$.
If $n\gr2(k-1)$, the estimate \eqref{eq:master-polynomial-absolute} gives the stronger bound without the exponential factor.
In general, \eqref{eq:master-polynomial-exact} gives
$$ \Pp\left\{R(C\cap F)>\frac C{\sqrt n}\exp\left(\frac{Ck}{\Delta}\right)\tau\right\}\ls\tau^{-\Delta}. $$
Indeed, $n=(k-1)+\Delta$, so the factor $2^{4(k-1)+n+1}$ in \eqref{eq:master-polynomial-exact} is absorbed after taking the $\Delta$-th root by $C\exp(Ck/\Delta)$.
Since $K^\circ\subseteq\lambda_nC$,
$$ R(K^\circ\cap F)\ls\lambda_nR(C\cap F). $$
Using $(P_FK)^\circ_F=K^\circ\cap F$ and $r(P_FK)=R(K^\circ\cap F)^{-1}$ proves \eqref{eq:projection-tail-all-dimensional}.
Writing $d=n-k$, the remaining assertions follow because $k/(d+1)$ is bounded in the stated ranges.
\end{proof}

A second estimate follows by applying Gordon's theorem to the set of directions on which the support function is small.
For a closed set $T\subseteq S^{n-1}$ put
$$ \omega(T)=\E\sup_{\theta\in T}\langle G,\theta\rangle, $$
and, for $t>0$,
$$ T_t(K)=\{\theta\in S^{n-1}:h_K(\theta)\ls t\sqrt n\}, \qquad \omega_K(t)=\omega(T_t(K)). $$

\begin{proposition}\label{prop:localized-width-reduction}
Let $K\subseteq\R^n$ contain the origin, let $1\ls d<n$, and let $F$ be Haar-distributed on $G_{n,n-d}$.
If $\omega_K(t)\ls c_0\sqrt d$ for a sufficiently small absolute constant $c_0>0$, then
$$ P_FK\supseteq t\sqrt nB_F $$
with probability at least $1-Ce^{-cd}$.
If $K$ is isotropic, then
\begin{equation}\label{eq:global-localized-width}
 \omega_K(t)\ls CtnM(K)\ls Ct\sqrt{n\ln(en)},
\end{equation}
and consequently
$$ P_FK\supseteq c\sqrt{\frac d{\ln(en)}}B_F $$
with the same probability estimate.
\end{proposition}

\begin{proof}
Gordon's theorem~\cite[Corollary~3.4]{Gordon-1988} implies that $F\cap T_t(K)=\varnothing$ outside a set of probability at most $Ce^{-cd}$ whenever $\omega_K(t)\ls c_0\sqrt d$.
On this event, $h_{P_FK}=h_K$ on $F$ and hence $P_FK\supseteq t\sqrt nB_F$.

For isotropic $K$, put $L_t=t\sqrt nK^\circ\cap B_2^n$.
Then $T_t(K)\subseteq L_t$ and
$$ \omega_K(t)\ls\E h_{L_t}(G) \ls t\sqrt n\,\E p_K(G) \ls CtnM(K). $$
The estimate $M(K)\ls C\sqrt{\ln(en)/n}$ gives \eqref{eq:global-localized-width}; take $t=c\sqrt{d/(n\ln(en))}$.
\end{proof}

Combining Theorem~\ref{th:projection-tail} and Proposition~\ref{prop:localized-width-reduction}, a Haar-distributed $F\in G_{n,n-d}$ satisfies
\begin{equation}
 r(P_FK)\gr c\sqrt n\max\left\{ \exp\left(-\frac{C(n-d)}{d+1}\right), \sqrt{\frac{d}{n\ln(en)}} \right\}
\end{equation}
with probability at least $1-Ce^{-cd}$.

For the cube, the Gaussian width of the set $T_t(Q_n)$ is computed in Lemma~\ref{lem:l1-spherical-width} in Appendix~\ref{appendix:special-calculations}.

\begin{proposition}\label{prop:cube-localized-width}
Let $Q_n=[-1/2,1/2]^n$.
For $ \frac1{2\sqrt n}\ls t\ls\frac12, $ one has
$$ \omega_{Q_n}(t) \simeq t\sqrt{n\ln\left(e+\frac1{t^2}\right)}. $$
For $0<t<1/(2\sqrt n)$, the set $T_t(Q_n)$ is empty.
\end{proposition}

The proof is given in Appendix~\ref{appendix:special-calculations}.

\begin{corollary}\label{cor:cube-projection}
Let $Q_n=[-1/2,1/2]^n$, let $1\ls d<n$, and let $F$ be Haar-distributed on $G_{n,n-d}$.
Then
$$ P_FQ_n\supseteq c\sqrt{\frac{d}{\ln(en/d)}}B_F $$
with probability at least $1-Ce^{-cd}$.
\end{corollary}

\begin{proof}
Put
$$ t=c_1\sqrt{\frac{d/n}{\ln(en/d)}}. $$
If $t<1/(2\sqrt n)$, then $T_t(Q_n)=\varnothing$ and the conclusion is immediate.
Otherwise, Proposition~\ref{prop:cube-localized-width} gives
$$ \omega_{Q_n}(t)^2 \ls Ct^2n\ln\left(e+\frac1{t^2}\right) \ls Cc_1^2d. $$
Choose $c_1$ sufficiently small and apply Proposition~\ref{prop:localized-width-reduction}.
\end{proof}

The same argument applies to every body which contains a coordinate cube.

\begin{proposition}\label{prop:cube-dominated-projection-profile}
Let $K\subseteq\R^n$ be a convex body containing the origin, assume that
$$ K\supseteq aB_\infty^n $$
for some $a>0$, let $1\ls d<n$, and let $F$ be Haar-distributed on $G_{n,n-d}$.
Then
\begin{equation}\label{eq:cube-dominated-projection-profile}
 r(P_FK)\gr ca\max\left\{1, \sqrt{\frac d{\ln(en/d)}}\right\}
\end{equation}
with probability at least $1-Ce^{-cd}$.
\end{proposition}

The proof is given in Appendix~\ref{appendix:special-calculations}.

For a volume-one isotropic body $K$ and Haar-distributed $F\in G_{n,n-d}$ put
$$ Y_K(F)=\sqrt n\,R(K^\circ\cap F). $$
For $1\ls d<n$ and $q>0$ define
$$ \mathfrak A_{n,d}(q)=\sup_K\|Y_K\|_{L_{q,\infty}}, $$
where the $L_{q,\infty}$ quasi-norm is taken with respect to the randomness of $F$ and the supremum is over all volume-one isotropic bodies.

The cube gives a deterministic lower bound for $\mathfrak A_{n,d}(q)$.
The relevant width estimate goes back to Kashin~\cite{Kashin-1975} and Garnaev--Gluskin~\cite{Garnaev-Gluskin-1984}; we use the sharp formulation \cite[Theorem~1.1]{Foucart-Pajor-Rauhut-Ullrich-2010}.

\begin{proposition}\label{prop:cube-gelfand-obstruction}
For every $1\ls d<n$, every $q>0$ and every subspace $F\subseteq\R^n$ of codimension $d$,
\begin{equation}\label{eq:cube-gelfand-obstruction}
 \sqrt n\,R(Q_n^\circ\cap F) \gr c\sqrt n\min\left\{1, \sqrt{\frac{\ln(en/d)}d}\right\}.
\end{equation}
Consequently,
\begin{equation}\label{eq:weak-profile-cube-lower}
 \mathfrak A_{n,d}(q) \gr c\sqrt n\min\left\{1, \sqrt{\frac{\ln(en/d)}d}\right\}.
\end{equation}
\end{proposition}

The proof is given in Appendix~\ref{appendix:special-calculations}.

We next prove Theorem~\ref{th:unconditional-extremal-profile}.
The cube containment used below goes back to Bobkov and Nazarov~\cite{Bobkov-Nazarov-2003}; we use the formulation in~\cite[Proposition~2.3]{Giannopoulos-Hartzoulaki-Tsolomitis-2005}.

\begin{proof}[Proof of Theorem~$\ref{th:unconditional-extremal-profile}$]
Let $K\in\mathcal I_n^{\rm unc}$.
By~\cite[Proposition~2.3]{Giannopoulos-Hartzoulaki-Tsolomitis-2005},
\begin{equation}\label{eq:unconditional-cube-inclusion}
 K\supseteq cB_\infty^n.
\end{equation}
Proposition~\ref{prop:cube-dominated-projection-profile} now gives
$$ r(P_FK)\gr c\max\left\{1, \sqrt{\frac d{\ln(en/d)}}\right\} $$
with probability at least $1-Ce^{-cd}$.
For all sufficiently large $d$ this probability is at least $1/2$; for bounded $d$, the deterministic inclusion \eqref{eq:unconditional-cube-inclusion} gives the same conclusion after changing the absolute constant.
This proves the lower estimate in \eqref{eq:unconditional-extremal-profile-estimate}.

For the reverse estimate, take $K=Q_n$.
Proposition~\ref{prop:cube-gelfand-obstruction} gives, for every codimension-$d$ subspace $F$,
$$ R(Q_n^\circ\cap F) \gr c\min\left\{1, \sqrt{\frac{\ln(en/d)}d}\right\}. $$
Since $r(P_FQ_n)=R(Q_n^\circ\cap F)^{-1}$, we obtain
$$ r(P_FQ_n) \ls C\max\left\{1, \sqrt{\frac d{\ln(en/d)}}\right\} $$
for every $F$.
This proves the upper estimate and completes the proof.
\end{proof}

The comparison between $\mathfrak A_{n,d}(d)$ and $\mathfrak A_{n,d}(d+1)$ is given in Appendix~\ref{appendix:special-calculations}.

\section{Random intersections and independent rotations}\label{section:4}

Comparisons between sections and intersections of a few rotations go back to V.~D.~Milman~\cite{Milman-1991} and Giannopoulos and V.~D.~Milman~\cite{Giannopoulos-Milman-1997}.
We first recall two known estimates for the intersection of two random rotations.
A result of Klartag and E.~Milman~\cite{Klartag-EMilman-2012b}, together with \cite[Theorem~3.3]{GPT-survey-2025}, implies that for every symmetric isotropic convex body $K$ a Haar rotation $U$ satisfies
$$ K^\circ\cap U(K^\circ)\subseteq\frac C{\sqrt n}B_2^n $$
with probability at least $1-\exp(-cn)$.
The local-to-global theorem of Rudelson and Vershynin~\cite[Appendix]{Vershynin-2006}, combined with the low $M^\ast$ estimate, gives the related bound
$$ R(K^\circ\cap U(K^\circ))\ls CM(K)\ls C\sqrt{\frac{\ln(en)}n} $$
with the same type of probability, where the last inequality follows from \eqref{eq:isotropic-spherical-means}.
Related estimates may be found in \cite{Milman-1991,Giannopoulos-Milman-Tsolomitis-2005,Litvak-Pajor-Tomczak-2006,Brazitikos-Stavrakakis-2014}.
Our proof of Theorem~\ref{th:m-rotation-main} is different and starts from an exact radial-tail identity.

A star body is a compact set, star shaped about the origin, with positive continuous radial function.
For such a body $C$ put
$$ \alpha_C(t)=\sigma\{\theta:\rho_C(\theta)\gr t\}. $$

\begin{lemma}
Let $C_1,\ldots,C_m$ be star bodies, let $U_1,\ldots,U_m$ be independent Haar rotations and put $D_m=\bigcap_{j=1}^mU_jC_j$.
If
$$ \alpha_{C_j}(t)=\sigma\{\theta:\rho_{C_j}(\theta)\gr t\}, $$
then, for every $p>0$,
\begin{equation}\label{eq:exact-radial-tail-identity}
 \E\int_{S^{n-1}}\rho_{D_m}(\theta)^p\,d\sigma(\theta) =p\int_0^\infty t^{p-1}\prod_{j=1}^m\alpha_{C_j}(t)\,dt.
\end{equation}
The same identity holds with $U_1=I_n$.
In particular, if the $C_j$ are polars of centered volume-one bodies, then, for $0<p<mn$,
$$ \left(\E\int_{S^{n-1}}\rho_{D_m}^p\,d\sigma\right)^{1/p} \ls\frac C{\sqrt n}\left(\frac{mn}{mn-p}\right)^{1/p}. $$
\end{lemma}

\begin{proof}
If all rotations are random, then for fixed $\theta$, independence and rotational invariance give
$$ \Pp\{\rho_{D_m}(\theta)\gr t\}=\prod_{j=1}^m\alpha_{C_j}(t). $$
Layer cake and Tonelli prove \eqref{eq:exact-radial-tail-identity}.
If $U_1=I_n$, the pointwise probability equals
$$ \mathds{1}_{\{\rho_{C_1}(\theta)\gr t\}}\prod_{j=2}^m\alpha_{C_j}(t), $$
and integration in $\theta$ gives the same product.
For centered volume-one bodies, \eqref{eq:intro-small-support} gives $\alpha_{C_j}(t)\ls\min\{1,(C/(t\sqrt n))^n\}$, and integration yields the last estimate.
\end{proof}

\begin{theorem}\label{th:mixed-rotation-projection}
There exist absolute constants $c,C>0$ such that the following holds.
Let $n\gr2$, $m\gr2$, let $K_1,\ldots,K_m\subseteq\R^n$ be isotropic convex bodies, and let $U_1=I_n,U_2,\ldots,U_m$ be independent Haar rotations.
Independently, let $F$ be Haar-distributed on $G_{n,k}$, where $1\ls k\ls n$.
Put
$$ H_m=\conv\{U_1K_1,\ldots,U_mK_m\}, \qquad q_{m,k}=mn-k+1. $$
Then, for every $\tau\gr1$,
\begin{equation}\label{eq:mixed-rotation-projection}
 \Pp\left\{P_FH_m\not\supseteq\frac{c\sqrt n}{\tau}B_F\right\}
 \ls\tau^{-q_{m,k}}.
\end{equation}
Equivalently,
$$ \Pp\left\{R\left(F\cap\bigcap_{j=1}^mU_jK_j^\circ\right)>\frac{C\tau}{\sqrt n}\right\} \ls\tau^{-q_{m,k}}. $$
\end{theorem}

\begin{proof}
For every $j$, let $\widehat K_j$ be the body given by Lemma~\ref{lem:recentered-core}, and put
$$ C_j=\widehat K_j^\circ, \qquad \widehat D_m=\bigcap_{j=1}^mU_jC_j. $$
Then
$$ \widehat D_m\supseteq\frac c{\sqrt n}B_2^n $$
and
$$ \alpha_{C_j}(t)\ls\min\left\{1,\left(\frac C{t\sqrt n}\right)^n\right\}. $$
For $t>0$ define
$$ Z_t(U,F)=\sigma_F\{\theta\in S_F:\rho_{\widehat D_m}(\theta)\gr t\}. $$
The incidence identity \eqref{eq:grassmann-incidence}, independence and rotational invariance give
$$ \E_{U,F}Z_t(U,F)=\prod_{j=1}^m\alpha_{C_j}(t) \ls\min\left\{1,\left(\frac C{t\sqrt n}\right)^{mn}\right\}. $$
The same formula is valid with $U_1=I_n$, because the averaging over $F$ makes the first sampled direction uniform on $S^{n-1}$.

Apply Corollary~\ref{cor:master-polynomial} to the pair $(\widehat D_m,F)$ with $ a\simeq r\simeq n^{-1/2}$, $\beta=mn$ and $d=k-1.$
Then $\Delta=q_{m,k}$.
Since $m\gr2$ and $k\ls n$ then $ mn\gr2(k-1), $ so the absolute-constant conclusion \eqref{eq:master-polynomial-absolute} applies.
Therefore
$$ \Pp\left\{R(\widehat D_m\cap F)>\frac{C\tau}{\sqrt n}\right\} \ls\tau^{-q_{m,k}}. $$
By \eqref{eq:recentered-core-polar-comparison},
$$ \bigcap_{j=1}^mU_jK_j^\circ \subseteq\lambda_n\widehat D_m, $$
so the same estimate holds, after changing the absolute constant, for $F\cap\bigcap_jU_jK_j^\circ$.

Finally,
$$ H_m^\circ=\bigcap_{j=1}^mU_jK_j^\circ, \qquad (P_FH_m)^\circ_F=H_m^\circ\cap F. $$
Taking polars in $F$ proves \eqref{eq:mixed-rotation-projection}.
\end{proof}

Taking $k=n$ in Theorem~\ref{th:mixed-rotation-projection} proves Theorem~\ref{th:m-rotation-main}.
In particular, for every fixed $\tau_0>1$ the exceptional probability is at most $\exp[-c_{\tau_0}((m-1)n+1)]$, where $c_{\tau_0}>0$ depends only on $\tau_0$.
Integration of the tail gives the corresponding circumradius moments for every $0<p<mn-k+1$.

The following isotropic cylinder shows that the cap-incidence exponent has the correct dimensional order.
Put
$$ Z_n=[-\sqrt3,\sqrt3]\times\sqrt{n+1}\,B_2^{n-1}, \qquad K_n=|Z_n|^{-1/n}Z_n. $$
Then $K_n$ is isotropic, $|Z_n|^{1/n}\simeq1$, and, writing $\theta=(\theta_1,\theta')$,
\begin{equation}
 \rho_{K_n^\circ}(\theta)
 \simeq\frac1{|\theta_1|+\sqrt n\,|\theta'|}.
\end{equation}
Consequently, there exist absolute constants $A,c,C>0$ such that, for $A\ls\tau\ls c\sqrt n$,
\begin{equation}\label{eq:cylinder-high-radius-caps}
 \left\{\operatorname{dist}(\theta,\{e_1,-e_1\})\ls\frac c\tau\right\} \subseteq\left\{\rho_{K_n^\circ}(\theta)\gr\frac{c\tau}{\sqrt n}\right\} \subseteq\left\{\operatorname{dist}(\theta,\{e_1,-e_1\})\ls\frac C\tau\right\}.
\end{equation}

\begin{proposition}
Let $m\gr2$, let $U_1=I_n,U_2,\ldots,U_m$ be independent Haar rotations, and put
$$ D_{m,n}=\bigcap_{j=1}^mU_jK_n^\circ. $$
There are absolute constants $A,c,c_0,C>0$ such that, for $A\ls\tau\ls c_0\sqrt n$,
$$ \left(\frac c\tau\right)^{(m-1)(n-1)} \ls\Pp\left\{R(D_{m,n})\gr\frac{c\tau}{\sqrt n}\right\} \ls\left(\frac C\tau\right)^{(m-1)(n-1)}. $$
\end{proposition}

\begin{proof}
For the lower estimate, require $U_j^{-1}e_1$ to lie within $c/\tau$ of $e_1$ for every $j\gr2$ and use the first inclusion in \eqref{eq:cylinder-high-radius-caps}.
For the upper estimate, if $R(D_{m,n})\gr c\tau/\sqrt n$, then some direction lies simultaneously in the corresponding high-radius caps, so $U_je_1$ lies within $C/\tau$ of $\{e_1,-e_1\}$ for every $j\gr2$.
For $0<a\ls c_0$, a spherical cap of Euclidean radius $a$ has measure between $(ca)^{n-1}$ and $(Ca)^{n-1}$; independence completes the proof.
\end{proof}

The endpoint argument gives the exponent $(m-1)n+1$, while the cylinder has exponent $(m-1)(n-1)$; for fixed $m$, the two have the same leading dimensional order.

\section{Small deviations, negative moments and Minkowski averages}\label{section:5}

We begin by recalling the small-deviation information needed below.
For a symmetric convex body $K$ put
$$ d(K)=d_\ast(K^\circ,1/2)=\min\{-\ln\sigma\{\theta:2p_K(\theta)\ls M(K)\},n\}, $$
and
$$ \beta(K)=\frac{\operatorname{Var}p_K(G)}{(\E p_K(G))^2},\qquad \beta_\ast(K)=\beta(K^\circ)=\frac{\operatorname{Var}h_K(G)}{(\E h_K(G))^2}. $$
We shall also use the critical dimension
$$ k(K)=n\left(\frac{M(K)}{b(K)}\right)^2. $$
Paouris and Valettas~\cite[Proposition~4.3]{Paouris-Valettas-variance} proved the general comparison
\begin{equation}\label{eq:beta-critical-dimension}
 \ln\frac{1}{\beta(K)}\ls Ck(K)\ls\frac{C}{\beta(K)}.
\end{equation}
The estimates of Klartag--Vershynin~\cite{Klartag-Vershynin-2007}, Paouris--Valettas~\cite{Paouris-Valettas-small-deviation} and Paouris--Pivovarov--Valettas~\cite{Paouris-Pivovarov-Valettas-Alexandrov} imply, for absolute constants $c,C>0$,
\begin{equation}\label{eq:negative-moment-d}
 M_{-q}(K)\gr cM(K)\qquad(0<q<cd(K)),
\end{equation}
and $d(K)\gr c/\beta(K)$.
More precisely, for $0<q<c/\beta(K)$,
$$ M_{-q}(K)\gr\left[1-C\min\left\{\frac{q}{k(K)},\max\{\sqrt{\beta(K)},q\beta(K)\}\right\}\right]M(K). $$
In particular,
\begin{equation}\label{eq:negative-moment-beta}
 M_{-q}(K)\gr cM(K)\qquad(0<q<c/\beta(K)).
\end{equation}
Equivalently, $w_{-q}(C)\gr cw(C)$ in the ranges $0<q<cd_\ast(C,1/2)$ and $0<q<c/\beta_\ast(C)$.

The use of random Minkowski averages as an empirical regularization goes back to Bourgain, Lindenstrauss and V.~Milman~\cite{BLM}.
Litvak, V.~D.~Milman and Schechtman~\cite{Litvak-Milman-Schechtman-1997,Litvak-Milman-Schechtman-1998} studied $q$-averages of norms and identified families of projective caps as a central geometric mechanism, while concentration and net implementations were developed by Artstein-Avidan, Friedland and V.~D.~Milman~\cite{Artstein-Friedland-Milman-2007}.
In particular, the bounded-distance threshold corresponding to $(b(K)/M(K))^2$ is classical.
The new point below is the quantitative estimate below this threshold: weighted negative moments provide the lower inclusion, while the upper inclusion follows the classical empirical method.

It is nevertheless false that $w_{-cn}(C)$ is always comparable to $w(C)$ for isotropic convex bodies.
The standard cross-polytope gives a simple example.

\begin{example}
Let $C_n=a_nB_1^n$, where
$$ a_n=|B_1^n|^{-1/n}=\left(\frac{n!}{2^n}\right)^{1/n}\simeq n. $$
Then $C_n$ is a symmetric isotropic convex body of volume one and
\begin{equation}\label{eq:cross-polytope-width}
 w(C_n)\simeq\sqrt{n\ln(n+1)}.
\end{equation}
For every fixed $\alpha\in(0,1)$,
\begin{equation}\label{eq:cross-polytope-negative}
 w_{-\alpha n}(C_n)\ls C_\alpha\sqrt n.
\end{equation}
Here $C_\alpha>0$ depends only on $\alpha$.
Consequently, $w_{-\alpha n}(C_n)/w(C_n)\to0$ as $n\to\infty$.
\end{example}

Indeed, if $G=(g_1,\ldots,g_n)$ is standard Gaussian and $\Theta=G/|G|$, then $\Theta$ is uniform on $S^{n-1}$, independent of $|G|$, and
$$ \E|G|\simeq\sqrt n,\qquad \E\|G\|_\infty\simeq\sqrt{\ln(n+1)}. $$
Hence
$$ \E\|\Theta\|_\infty\simeq\sqrt{\frac{\ln(n+1)}n}, $$
and \eqref{eq:cross-polytope-width} follows from $h_{C_n}(\theta)=a_n\|\theta\|_\infty$.
For \eqref{eq:cross-polytope-negative}, use the event $1\ls|g_i|\ls2$ for all $i$.
It has probability $p_0^n$ for an absolute $p_0\in(0,1)$ and implies $\|G\|_\infty/|G|\ls2/\sqrt n$.
Therefore
$$ \sigma\left\{\theta:\|\theta\|_\infty\ls\frac2{\sqrt n}\right\}\gr p_0^n, $$
which gives $w_{-\alpha n}(C_n)\ls C_\alpha\sqrt n$.

\smallskip 

Let $K_1,\ldots,K_m$ contain the origin, let $U_1,\ldots,U_m$ be independent Haar rotations and let $\lambda_j\gr0$, $\sum_j\lambda_j=1$.
Put
$$ A_\lambda=\sum_{j=1}^m\lambda_jU_j^*(K_j^\circ), \qquad b_j=b(K_j), \qquad b_\lambda=\sum_{j=1}^m\lambda_jb_j, $$
and, for $q>0$,
$$ a_q(\lambda)=\prod_{j=1}^mM_{-q\lambda_j}(K_j)^{\lambda_j}, $$
with the usual interpretation when $\lambda_j=0$.

\begin{theorem}\label{th:weighted-negative-profile}
Let the notation be as above.
Assume that $q\gr2n$ and put
$$ \Delta=q-n+1. $$
Then, for every $u\gr0$, with probability at least $1-e^{-u}$,
\begin{equation}\label{eq:weighted-profile-lower}
 A_\lambda\supseteq c\,a_q(\lambda) \left(\frac{c\,a_q(\lambda)}{b_\lambda}\right)^{\frac{n-1}{\Delta}} e^{-u/\Delta}B_2^n.
\end{equation}
Moreover, if
$$ \overline M_\lambda=\sum_{j=1}^m\lambda_jM(K_j), \qquad B_\lambda=\left(\sum_{j=1}^m\lambda_j^2b_j^2\right)^{1/2}, $$
then, with probability at least $1-e^{-cn}$,
\begin{equation}\label{eq:weighted-profile-upper}
 A_\lambda\subseteq C(\overline M_\lambda+B_\lambda)B_2^n.
\end{equation}
\end{theorem}

\begin{proof}
For fixed $\theta\in S^{n-1}$ put $ X_j=p_{K_j}(U_j\theta). $
Then
$$ h_{A_\lambda}(\theta)=\sum_{j=1}^m\lambda_jX_j. $$
By weighted arithmetic-geometric mean,
$$ h_{A_\lambda}(\theta)^{-q} \ls\prod_{j=1}^mX_j^{-q\lambda_j}. $$
Independence and rotational invariance give
$$ \E h_{A_\lambda}(\theta)^{-q} \ls\prod_{j=1}^m\left(\int_{S^{n-1}}p_{K_j}^{-q\lambda_j}\,d\sigma\right) =a_q(\lambda)^{-q}. $$
Let $L=A_\lambda^\circ$.
Since $h_{A_\lambda}\ls b_\lambda$, we have $ L\supseteq b_\lambda^{-1}B_2^n. $
Also,
$$ \E\sigma\{\rho_L\gr t\} =\E\sigma\{h_{A_\lambda}\ls t^{-1}\} \ls(ta_q(\lambda))^{-q}. $$
Apply Corollary~\ref{cor:master-polynomial} to the pair $(L,\R^n)$ with $ r=b_\lambda^{-1}$, $a=a_q(\lambda)^{-1}$, $\beta=q$ and $d=n-1. $
Since $q\gr2n>2(n-1)$, the absolute-constant form applies and yields
$$ \Pp\left\{R(L)>\frac C{a_q(\lambda)} \left(\frac{Cb_\lambda}{a_q(\lambda)}\right)^{\frac{n-1}{\Delta}} e^{u/\Delta}\right\} \ls e^{-u}. $$
Taking polars proves \eqref{eq:weighted-profile-lower}.

For the upper inclusion we follow the standard empirical concentration-and-net scheme; see~\cite{BLM,Artstein-Friedland-Milman-2007}.
Each $p_{K_j}$ is $b_j$-Lipschitz on the sphere.
For fixed $\theta$, tensorized spherical concentration gives
$$ \Pp\{|h_{A_\lambda}(\theta)-\overline M_\lambda|>t\} \ls2\exp\left(-\frac{cnt^2}{B_\lambda^2}\right). $$
Take $t=C_0B_\lambda$ and use a $1/2$-net $\mathcal N\subseteq S^{n-1}$ with $|\mathcal N|\ls5^n$.
Outside a set of probability at most $e^{-cn}$,
$$ h_{A_\lambda}(\theta)\ls\overline M_\lambda+C_0B_\lambda \qquad(\theta\in\mathcal N). $$
For every convex body $A$ containing the origin,
$$ R(A)\ls2\max_{\theta\in\mathcal N}h_A(\theta). $$
This proves \eqref{eq:weighted-profile-upper}.
\end{proof}

For a single body and equal weights put
$$ A_m=\frac1m\sum_{j=1}^mU_j^*(K^\circ). $$

\begin{proof}[Proof of Theorem~$\ref{th:minkowski-average-negative-moments}$]
Apply Theorem~\ref{th:weighted-negative-profile} with $K_j=K$, $\lambda_j=1/m$ and the order $q$ from the statement.
Then
$$ a_q(\lambda)=M_{-q/m}(K)\gr c_0M(K), \qquad b_\lambda=b(K), $$
and the lower inclusion is exactly \eqref{eq:minkowski-average-lower-inclusion}, after changing the absolute constant.
Moreover,
$$ \overline M_\lambda=M(K), \qquad B_\lambda=\frac{b(K)}{\sqrt m}, $$
so \eqref{eq:weighted-profile-upper} gives \eqref{eq:minkowski-average-upper-inclusion}.
Combining the two inclusions proves \eqref{eq:minkowski-average-geometric-distance}.
\end{proof}

\begin{corollary}\label{cor:minkowski-small-deviation}
Let $K$, $U_1,\ldots,U_m$ and $A_m$ be as in Theorem~$\ref{th:minkowski-average-negative-moments}$.
If $md(K)\gr Cn$, then, with probability at least $1-e^{-cn}$,
\begin{equation}\label{eq:minkowski-average-d-bound}
 d_{\rm G}(A_m,B_2^n)
 \ls C\left(1+\frac{b(K)}{M(K)\sqrt m}\right)
 \left(\frac{Cb(K)}{M(K)}\right)^{\frac{Cn}{md(K)}}.
\end{equation}
The same conclusion holds with the last factor replaced by
$$ \left(\frac{Cb(K)}{M(K)}\right)^{Cn\beta(K)/m} $$
under the assumption $m/\beta(K)\gr Cn$.
\end{corollary}

\begin{proof}
Use \eqref{eq:negative-moment-d}, respectively \eqref{eq:negative-moment-beta}, in Theorem~\ref{th:minkowski-average-negative-moments}, choose the order to be a sufficiently small multiple of $md(K)$, respectively $m/\beta(K)$, and take $u=cn$.
\end{proof}

The same argument gives a projected estimate without changing the negative-moment input.

\begin{theorem}\label{th:projected-weighted-negative-profile}
Let $K_1,\ldots,K_m$ contain the origin, let $U_1,\ldots,U_m$ be independent Haar rotations and let $\lambda_j\gr0$, $\sum_j\lambda_j=1$.
Put
$$ A_\lambda=\sum_{j=1}^m\lambda_jU_j^*(K_j^\circ), \qquad b_\lambda=\sum_{j=1}^m\lambda_jb(K_j), \qquad a_q(\lambda)=\prod_{j=1}^mM_{-q\lambda_j}(K_j)^{\lambda_j}. $$
As before, the last expression has the usual interpretation when $\lambda_j=0$.
Independently, let $F$ be Haar-distributed on $G_{n,k}$, where $1\ls k\ls n$.
Assume that $q\gr2k$ and put $\Delta=q-k+1$.
Then, for every $u\gr0$, with probability at least $1-e^{-u}$,
\begin{equation}\label{eq:projected-weighted-profile-lower}
 P_FA_\lambda\supseteq c\,a_q(\lambda)\left(\frac{c\,a_q(\lambda)}{b_\lambda}\right)^{\frac{k-1}{\Delta}}e^{-u/\Delta}B_F.
\end{equation}
Moreover, if
$$ \overline M_\lambda=\sum_{j=1}^m\lambda_jM(K_j), \qquad B_\lambda=\left(\sum_{j=1}^m\lambda_j^2b(K_j)^2\right)^{1/2}, $$
then, with probability at least $1-e^{-ck}$,
\begin{equation}\label{eq:projected-weighted-profile-upper}
 P_FA_\lambda\subseteq C\left(\overline M_\lambda+\sqrt{\frac kn}\,B_\lambda\right)B_F.
\end{equation}
\end{theorem}

\begin{proof}
The calculation in the proof of Theorem~\ref{th:weighted-negative-profile} gives, for every fixed $\theta\in S^{n-1}$,
$$ \E_Uh_{A_\lambda}(\theta)^{-q}\ls a_q(\lambda)^{-q}. $$
Let $L=A_\lambda^\circ$.
Since $h_{A_\lambda}\ls b_\lambda$, we have $L\supseteq b_\lambda^{-1}B_2^n$.
For $t>0$ put
$$ Z_t(U,F)=\sigma_F\{\theta\in S_F:\rho_L(\theta)\gr t\}. $$
Since $F$ is independent of the rotations, the preceding negative-moment estimate and the incidence identity give
$$ \E_{U,F}Z_t(U,F)=\E_F\int_{S_F}\Pp_U\{h_{A_\lambda}(\theta)\ls t^{-1}\}\,d\sigma_F(\theta) \ls(ta_q(\lambda))^{-q}. $$
Apply Corollary~\ref{cor:master-polynomial} to the pair $(L,F)$ with $ r=b_\lambda^{-1}$, $a=a_q(\lambda)^{-1}$, $\beta=q$ and $d=k-1.$
Since $q\gr2k>2(k-1)$, we obtain
$$ \Pp\left\{R(L\cap F)>\frac C{a_q(\lambda)}\left(\frac{Cb_\lambda}{a_q(\lambda)}\right)^{\frac{k-1}{\Delta}}e^{u/\Delta}\right\}\ls e^{-u}. $$
Since $(P_FA_\lambda)^\circ_F=L\cap F$, taking polars in $F$ proves \eqref{eq:projected-weighted-profile-lower}.

For the upper inclusion, condition on $F$ and let $\mathcal N_F$ be a $1/2$-net of $S_F$ with $|\mathcal N_F|\ls5^k$.
The concentration estimate used above, with $t=C_0\sqrt{k/n}\,B_\lambda$, and a union bound give, outside a set of probability at most $e^{-ck}$,
$$ h_{A_\lambda}(\theta)\ls\overline M_\lambda+C_0\sqrt{\frac kn}\,B_\lambda \qquad(\theta\in\mathcal N_F). $$
Since
$$ R(P_FA_\lambda)\ls2\max_{\theta\in\mathcal N_F}h_{A_\lambda}(\theta), $$
this proves \eqref{eq:projected-weighted-profile-upper}.
\end{proof}

For equal bodies and equal weights we obtain the following consequence.

\begin{corollary}\label{cor:projected-minkowski-average-negative-moments}
Let $K\subseteq\R^n$ be a symmetric convex body, let $U_1,\ldots,U_m$ be independent Haar rotations and put
$$ A_m=\frac1m\sum_{j=1}^mU_j^*(K^\circ), \qquad M=M(K), \qquad b=b(K). $$
Independently, let $F$ be Haar-distributed on $G_{n,k}$, where $1\ls k\ls n$.
Assume that, for some $p_0>0$,
$$ M_{-p}(K)\gr c_0M\qquad(0<p\ls p_0). $$
Let $2k\ls q\ls mp_0$ and put $\Delta=q-k+1$.
Then, for every $u\gr0$, with probability at least $1-e^{-u}$,
\begin{equation}\label{eq:projected-minkowski-average-lower}
 P_FA_m\supseteq cM\left(\frac{cM}{b}\right)^{\frac{k-1}{\Delta}}e^{-u/\Delta}B_F.
\end{equation}
Moreover, with probability at least $1-e^{-ck}$,
\begin{equation}\label{eq:projected-minkowski-average-upper}
 P_FA_m\subseteq C\left(M+b\sqrt{\frac{k}{mn}}\right)B_F.
\end{equation}
Consequently, with probability at least $1-e^{-u}-e^{-ck}$,
\begin{equation}\label{eq:projected-minkowski-average-distance}
 d_{\rm G}(P_FA_m,B_F)\ls C\left(1+\frac bM\sqrt{\frac{k}{mn}}\right)\left(\frac{Cb}{M}\right)^{\frac{k-1}{\Delta}}e^{u/\Delta}.
\end{equation}
\end{corollary}

\begin{proof}
Apply Theorem~\ref{th:projected-weighted-negative-profile} with $K_j=K$ and $\lambda_j=1/m$.
Then
$$ a_q(\lambda)=M_{-q/m}(K)\gr c_0M(K), \qquad b_\lambda=b(K), \qquad \overline M_\lambda=M(K), \qquad B_\lambda=\frac{b(K)}{\sqrt m}. $$
The two inclusions and their combination give the assertion.
\end{proof}

\begin{corollary}\label{cor:projected-minkowski-small-deviation}
Let $K$, $U_1,\ldots,U_m$, $F$ and $A_m$ be as in Corollary~\ref{cor:projected-minkowski-average-negative-moments}.
If $md(K)\gr Ck$, then, with probability at least $1-e^{-ck}$,
\begin{equation}\label{eq:projected-minkowski-d-bound}
 d_{\rm G}(P_FA_m,B_F)\ls C\left(1+\frac{b(K)}{M(K)}\sqrt{\frac{k}{mn}}\right)\left(\frac{Cb(K)}{M(K)}\right)^{\frac{Ck}{md(K)}}.
\end{equation}
The same conclusion holds with the last factor replaced by
$$ \left(\frac{Cb(K)}{M(K)}\right)^{Ck\beta(K)/m} $$
under the assumption $m/\beta(K)\gr Ck$.
\end{corollary}

\begin{proof}
Use \eqref{eq:negative-moment-d}, respectively \eqref{eq:negative-moment-beta}, in Corollary~\ref{cor:projected-minkowski-average-negative-moments}, choose $q$ to be a sufficiently small multiple of $md(K)$, respectively $m/\beta(K)$, and take $u=ck$.
\end{proof}

The projected theorem also accepts inputs from $L_p$-centroid bodies.
Let $T\subseteq\R^n$ be isotropic and let $Z_p(T)$ denote its $L_p$-centroid body.
Letwin and Mikulincer~\cite[Proposition~2.4]{Letwin-Mikulincer-2026} proved that $w_{-2p}(Z_p(T))\simeq\sqrt p$ for $1\ls p\ls2c_0n$.
Moreover, $w(Z_p(T))\ls C\sqrt{p\ln(ep)}$ by~\cite[Lemma~3.4]{GPT-subgaussian-2026}.
Since $p_{Z_p(T)^\circ}=h_{Z_p(T)}$ and $R(Z_p(T))\ls Cp$, put $A_{m,p}=m^{-1}\sum_jU_j^*Z_p(T)$.
The two parts of Theorem~\ref{th:projected-weighted-negative-profile}, applied with $K_j=Z_p(T)^\circ$, equal weights and $q=2mp$, show, with $U_j$ and $F$ as above, that if $mp\gr k$, then, for every $u\gr0$,
\begin{align*}
 P_FA_{m,p}&\supseteq c\sqrt p\,p^{-\frac{k-1}{2(2mp-k+1)}}e^{-\frac{u}{2mp-k+1}}B_F,\\
 P_FA_{m,p}&\subseteq C\left(\sqrt{p\ln(ep)}+p\sqrt{\frac{k}{mn}}\right)B_F,
\end{align*}
with probability at least $1-e^{-u}$ and $1-e^{-ck}$, respectively.

For the cube, $M_{-p}(B_\infty^n)$ can be estimated for every $p>0$.
Chasapis, K\"onig and Tkocz~\cite{Chasapis-Konig-Tkocz-2021} obtained sharp negative-moment estimates for one-dimensional sums associated with cube sections.
The proposition below concerns instead the spherical $\ell_\infty$ gauge.

\begin{proposition}\label{prop:cube-negative-moment-profile}
There exist absolute constants $c,C>0$ such that, for every $n\gr2$ and every $p>0$,
\begin{equation}\label{eq:cube-negative-moment-profile}
 c\sqrt{\frac{\ln\left(e+\frac{n}{1+p}\right)}n} \ls M_{-p}(B_\infty^n) \ls C\sqrt{\frac{\ln\left(e+\frac{n}{1+p}\right)}n}.
\end{equation}
\end{proposition}

The proof is given in Appendix~\ref{appendix:special-calculations}.

The preceding proposition makes the projected estimate explicit for the cube.

\begin{corollary}\label{cor:projected-cube-profile}
Let $n,m\gr2$ and $1\ls k\ls n$.
Let $U_1,\ldots,U_m$ be independent Haar-distributed orthogonal transformations, let $F$ be an independent Haar-distributed element of $G_{n,k}$, and put
$$ A_m=\frac1m\sum_{j=1}^mU_j^*(Q_n^\circ), \qquad Q_n=[-1/2,1/2]^n. $$
For $p\gr2k/m$ set
$$ \ell_p=\ln\left(e+\frac{n}{1+p}\right), \qquad \Delta_p=mp-k+1. $$
Then, with probability at least $1-Ce^{-ck}$,
\begin{equation}\label{eq:projected-cube-profile-general}
 d_{\rm G}(P_FA_m,B_F)\ls C\left(\sqrt{\frac{\ln(en)}{\ell_p}}+\sqrt{\frac{k}{m\ell_p}}\right)\left(C\sqrt{\frac n{\ell_p}}\right)^{\frac{k-1}{\Delta_p}}\exp\left(\frac{Ck}{\Delta_p}\right).
\end{equation}
In particular, if
\begin{equation}\label{eq:projected-cube-logarithmic-scale}
 L_{m,k}=\ln\left(e+\min\left\{n,\frac{mn}{k\ln(en)}\right\}\right),
\end{equation}
then
\begin{equation}\label{eq:projected-cube-profile-optimized}
 d_{\rm G}(P_FA_m,B_F)\ls C\left(\sqrt{\frac{\ln(en)}{L_{m,k}}}+\sqrt{\frac{k}{mL_{m,k}}}\right)
\end{equation}
with probability at least $1-Ce^{-ck}$.
\end{corollary}

The proof is given in Appendix~\ref{appendix:special-calculations}.

\section{Isomorphic global Dvoretzky theorem for isotropic convex bodies}\label{section:6}

We use two standard estimates.
Let $\overline B_2^n=\omega_n^{-1/n}B_2^n$ be the Euclidean ball of volume one.
The reverse Brunn--Minkowski estimate for isotropic convex bodies gives
\begin{equation}\label{eq:reverse-BM-isotropic}
 |K+s\overline B_2^n|^{1/n}\ls C(1+s),\qquad s\gr0,
\end{equation}
for every isotropic $K$; see \cite[Theorem~3.9]{GPT-survey-2025}.
We shall also use the theorem of Bourgain, Lindenstrauss and V.~Milman~\cite{BLM} in the Chernoff form of Artstein-Avidan, Friedland and V.~D.~Milman~\cite{Artstein-Friedland-Milman-2007}: if $K$ is symmetric, $M=M(K)$, $b=b(K)$, $0<\varepsilon<1/2$ and
$$ k\gr\frac{c_1}{\varepsilon^2}\left(\frac bM\right)^2, $$
then independent Haar rotations $U_1,\ldots,U_k$ satisfy
\begin{equation}\label{eq:BLM-norm-comparison}
 \frac{M}{1+\varepsilon}|x|\ls\frac1k\sum_{i=1}^k\|U_ix\|_K\ls(1+\varepsilon)M|x|\qquad(x\in\R^n)
\end{equation}
with probability at least
$$ 1-\exp\left(-c_2\varepsilon^2nk\left(\frac Mb\right)^2\right). $$

The regularization by adjoining a Euclidean ball follows Fresen~\cite{Fresen-2015} and Chasapis--Giannopoulos~\cite{Chasapis-Giannopoulos-2016}.
We compare $K$ with a body obtained by adjoining a Euclidean ball.
In the isotropic setting the inradius determines the value of the parameter $t$ used in the comparison below.

\begin{proof}[Proof of Theorem~$\ref{th:global-dvoretzky}$]
Let $r=r(K)$, so $rB_2^n\subseteq K$.
For $t\gr1$ set
$$ K_t=\conv(K\cup trB_2^n),\qquad M_t=\int_{S^{n-1}}\|x\|_{K_t}\,d\sigma(x),\qquad b_t=\max_{S^{n-1}}\|x\|_{K_t}. $$
Since $rB_2^n\subseteq K$,
\begin{equation}\label{eq:isotropic-comparison}
 K\subseteq K_t\subseteq tK, \qquad \frac1t\|x\|_K\ls\|x\|_{K_t}\ls\|x\|_K.
\end{equation}
Also, $h_{K_t}=\max\{h_K,tr\}$ and $\min_{S^{n-1}}h_K=r$, so $r(K_t)=tr$ and
$$ b_t=\frac1{tr}. $$
By \eqref{eq:polar-integration} and H\"older's inequality,
$$ M_t\gr\left(\int_{S^{n-1}}\|x\|_{K_t}^{-n}\,d\sigma(x)\right)^{-1/n} =\frac{\omega_n^{1/n}}{|K_t|^{1/n}}\gr\frac c{\sqrt n\,|K_t|^{1/n}}. $$
Moreover,
$$ K_t\subseteq K+trB_2^n\subseteq K+C\frac{tr}{\sqrt n}\overline B_2^n, $$
and \eqref{eq:reverse-BM-isotropic} gives
$$ |K_t|^{1/n}\ls C\max\left\{1,\frac{tr}{\sqrt n}\right\}. $$
Consequently,
\begin{equation}\label{eq:isotropic-ratio}
 \left(\frac{b_t}{M_t}\right)^2\ls C\max\left\{\frac n{t^2r^2},1\right\}.
\end{equation}

Fix $\varepsilon=1/4$ and set
$$ t_k=\max\left\{1,\frac A r\sqrt{\frac nk}\right\}, $$
where $A$ is a sufficiently large absolute constant.
Then $n/(t_k^2r^2)\ls k/A^2$, and \eqref{eq:isotropic-ratio} shows, after choosing $A$ and then an absolute $k_0$, that the hypothesis of the Bourgain--Lindenstrauss--Milman theorem is satisfied for every $k\gr k_0$.
Thus \eqref{eq:BLM-norm-comparison} applied to $K_{t_k}$ gives, with probability at least $1-e^{-c'n}$,
$$ \frac45M_{t_k}|x|\ls\frac1k\sum_{i=1}^k\|U_i x\|_{K_{t_k}}\ls\frac54M_{t_k}|x|. $$
Using \eqref{eq:isotropic-comparison},
\begin{equation}\label{eq:isotropic-transfer}
 \frac45M_{t_k}|x|\ls\frac1k\sum_{i=1}^k\|U_i x\|_K\ls\frac{5t_k}{4}M_{t_k}|x|.
\end{equation}
Finally,
$$ h_{\frac1k\sum_{i=1}^kU_i^*(K^\circ)}(x)=\frac1k\sum_{i=1}^k\|U_i x\|_K. $$
It follows from \eqref{eq:isotropic-transfer} that
$$ d_{\rm G}\left(\frac1k\sum_{i=1}^kU_i^*(K^\circ),B_2^n\right)\ls2t_k\ls C\max\left\{1,\frac1{r(K)}\sqrt{\frac nk}\right\}, $$
which proves \eqref{eq:isotropic-global}.
\end{proof}

Let $K\subseteq\R^n$ be a symmetric isotropic convex body, let $m\gr k_0$, and let $U_1,\ldots,U_m$ be independent Haar rotations.
Put
$$ A_m=\frac1m\sum_{j=1}^mU_j^*(K^\circ). $$
Combining Theorem~\ref{th:global-dvoretzky} with the estimate of Section~\ref{section:5}, we obtain, for every $p\gr2n/m$, outside an event of probability at most $2e^{-cn}$,
$$ d_{\rm G}(A_m,B_2^n) \ls C\min\left\{ \max\left\{1,\frac1{r(K)}\sqrt{\frac nm}\right\}, \frac{M(K)+1/(r(K)\sqrt m)} {M_{-p}(K)\left(c r(K)M_{-p}(K)\right)^{\frac{n-1}{mp-n+1}}e^{-cn/(mp-n+1)}} \right\}. $$

For the cube this gives an improvement over the inradius-dependent estimate.
Put
$$ Q_n=[-1/2,1/2]^n, \qquad L_m=\ln\left(e+\min\left\{n,\frac{m}{\ln(en)}\right\}\right). $$

\begin{corollary}\label{cor:cube-global-regularization}
Let $m\gr2$, let $U_1,\ldots,U_m$ be independent Haar-distributed orthogonal transformations and put
$$ A_m=\frac1m\sum_{j=1}^mU_j^*(Q_n^\circ). $$
Then, with probability at least $1-e^{-cn}$,
\begin{equation}\label{eq:cube-profile-distance}
 d_{\rm G}(A_m,B_2^n)
 \ls C\left( \sqrt{\frac{\ln(en)}{L_m}} +\sqrt{\frac n{mL_m}} \right).
\end{equation}
In particular, if $m\gr Cn/\ln(en)$, then
$$ d_{\rm G}(A_m,B_2^n)\ls C. $$
\end{corollary}

\begin{proof}
This is also the case $k=n$ of Corollary~\ref{cor:projected-cube-profile}; we retain the direct argument for the full-dimensional statement.
Since $p_{Q_n}=2\|\cdot\|_\infty$, Proposition~\ref{prop:cube-negative-moment-profile} gives
$$ M_{-p}(Q_n) \simeq\sqrt{\frac{\ln\left(e+\frac{n}{1+p}\right)}n}, \qquad M(Q_n)\simeq\sqrt{\frac{\ln(en)}n}, \qquad b(Q_n)=2. $$
Choose $p=A n\ln(en)/m$, with $A$ a sufficiently large absolute constant.
Then $mp\gr2n$ and
$$ M_{-p}(Q_n)\gr c\sqrt{\frac{L_m}{n}}. $$
Moreover, $(n-1)/(mp-n+1)\ls C/\ln(en)$, and hence
$$ \left(c r(Q_n)M_{-p}(Q_n)\right)^{\frac{n-1}{mp-n+1}}\gr c,\qquad e^{-cn/(mp-n+1)}\gr c. $$
Substitution gives \eqref{eq:cube-profile-distance}; the last assertion follows from $L_m\simeq\ln(en)$ in the stated range.
\end{proof}

Since $r(Q_n)=1/2$, Theorem~\ref{th:global-dvoretzky} alone gives bounded geometric distance only for $m\gr Cn$.
Thus Corollary~\ref{cor:cube-global-regularization} improves the number of rotations by the factor $\ln(en)$.

The following deterministic observation gives a lower estimate for every number of rotations.

\begin{proposition}\label{prop:rotation-average-obstruction}
Let $K\subseteq\R^n$ be a symmetric convex body containing the origin, let $m\gr1$, let $U_1,\ldots,U_m\in O(n)$ and put
$$ A_m=\frac1m\sum_{j=1}^mU_j^*(K^\circ). $$
Then
\begin{equation}\label{eq:rotation-average-obstruction}
 d_{\rm G}(A_m,B_2^n)\gr\max\left\{1,\frac{b(K)}{M(K)\sqrt m}\right\}.
\end{equation}
In particular, for the cube,
\begin{equation}\label{eq:cube-profile-lower}
 d_{\rm G}(A_m,B_2^n)\gr c\max\left\{1,\sqrt{\frac{n}{m\ln(en)}}\right\}.
\end{equation}
\end{proposition}

\begin{proof}
For every $j$ choose $x_j\in U_j^*(K^\circ)$ with $|x_j|=b(K)=R(K^\circ)$.
Since $K^\circ$ is symmetric,
$$ \frac1m\sum_{j=1}^m\varepsilon_jx_j\in A_m $$
for every choice of signs $\varepsilon_j\in\{-1,1\}$.
On averaging over the signs,
$$ \E_\varepsilon\left|\sum_{j=1}^m\varepsilon_jx_j\right|^2=\sum_{j=1}^m|x_j|^2=mb(K)^2, $$
and hence $R(A_m)\gr b(K)/\sqrt m$.
On the other hand, rotation invariance and additivity of the mean width give $ w(A_m)=w(K^\circ)=M(K), $ so $r(A_m)\ls M(K)$.
Since $A_m$ is symmetric, $d_{\rm G}(A_m,B_2^n)=R(A_m)/r(A_m)$, and \eqref{eq:rotation-average-obstruction} follows.
The cube estimate follows from $b(Q_n)=2$ and $M(Q_n)\simeq\sqrt{\ln(en)/n}$.
\end{proof}

It follows that bounded geometric distance for the cube requires $m\gr cn/\ln(en)$.
The matching sufficiency also follows from the classical theorem of Litvak--Milman--Schechtman~\cite{Litvak-Milman-Schechtman-1997,Litvak-Milman-Schechtman-1998}, since $(b(Q_n)/M(Q_n))^2\simeq n/\ln(en)$.
Thus the additional content of Corollary~\ref{cor:cube-global-regularization} is the quantitative estimate \eqref{eq:cube-profile-distance} below this threshold.

We finally show that the dependence on the inradius in Theorem~\ref{th:global-dvoretzky} is optimal throughout its possible range.

\begin{proposition}\label{prop:product-cylinder-global-sharpness}
Let $n\gr2$ and let $1\ls s\ls n/2$ be an integer.
Put
$$ Z_{n,s}=\sqrt{s+2}\,B_2^s\times\sqrt{n-s+2}\,B_2^{n-s}, \qquad K_{n,s}=|Z_{n,s}|^{-1/n}Z_{n,s}. $$
Then $K_{n,s}$ is symmetric and isotropic, and
\begin{equation}\label{eq:product-cylinder-parameters}
 r(K_{n,s})\simeq\sqrt s, \qquad b(K_{n,s})\simeq\frac1{\sqrt s}, \qquad M(K_{n,s})\simeq\frac1{\sqrt n}.
\end{equation}
Consequently, for every $m\gr1$ and every $U_1,\ldots,U_m\in O(n)$,
\begin{equation}\label{eq:product-cylinder-lower}
 d_{\rm G}\left(\frac1m\sum_{j=1}^mU_j^*(K_{n,s}^\circ),B_2^n\right)\gr c\max\left\{1,\sqrt{\frac{n}{sm}}\right\}.
\end{equation}
If $m\gr k_0$ and the rotations are independent and Haar-distributed, then, with probability at least $1-e^{-cn}$,
\begin{equation}\label{eq:product-cylinder-two-sided}
 d_{\rm G}\left(\frac1m\sum_{j=1}^mU_j^*(K_{n,s}^\circ),B_2^n\right)\simeq\max\left\{1,\sqrt{\frac{n}{sm}}\right\}.
\end{equation}
In particular, the threshold $m\simeq n/r(K)^2$ in Theorem~\ref{th:global-dvoretzky} is optimal, up to absolute constants, throughout the possible range of the inradius.
\end{proposition}

\begin{proof}
The uniform probability measure on $\sqrt{d+2}\,B_2^d$ has covariance matrix $I_d$.
It follows that the uniform probability measure on $Z_{n,s}$ has covariance matrix $I_n$, and hence its volume-one homothetic image $K_{n,s}$ is isotropic.
Moreover,
$$ |Z_{n,s}|^{1/n}=\left[\omega_s(s+2)^{s/2}\omega_{n-s}(n-s+2)^{(n-s)/2}\right]^{1/n}\simeq1. $$
Writing $\theta=(\theta',\theta'')\in S^{n-1}\subseteq\R^s\times\R^{n-s}$, we have
$$ p_{Z_{n,s}}(\theta)=\max\left\{\frac{|\theta'|}{\sqrt{s+2}},\frac{|\theta''|}{\sqrt{n-s+2}}\right\}. $$
This proves the first two estimates in \eqref{eq:product-cylinder-parameters}.
The same formula gives $p_{Z_{n,s}}(\theta)\gr1/\sqrt{n+4}$, while
$$ \int_{S^{n-1}}|\theta'|\,d\sigma(\theta)\ls\sqrt{\frac sn}, \qquad \int_{S^{n-1}}|\theta''|\,d\sigma(\theta)\ls\sqrt{\frac{n-s}{n}}. $$
Since the maximum is bounded by the sum, we obtain $M(K_{n,s})\simeq n^{-1/2}$.
Now \eqref{eq:product-cylinder-lower} follows from Proposition~\ref{prop:rotation-average-obstruction}, while the reverse estimate in \eqref{eq:product-cylinder-two-sided} follows from Theorem~\ref{th:global-dvoretzky}.
\end{proof}

\section{Volumes of random intersections}\label{section:7}
\bigskip

We finally turn to the proof of Theorem~\ref{th:unconditional-random-intersection-volume}.
The deterministic part is a consequence of polarity and the classical estimates for polytopes with few vertices.
The random upper estimate uses the complete radial profile of the cube.
We isolate the only estimate that is needed.

\begin{lemma}\label{lem:cube-random-intersection-tail}
There exist absolute constants $A,C>0$ such that the following holds.
Let $n,m\gr2$, put $\Lambda=\Lambda_{m,n}$ and
$$ a=A\sqrt{\frac n\Lambda},\qquad \alpha_n(t)=\sigma\left\{\theta\in S^{n-1}:\|\theta\|_\infty\ls\frac1{2t}\right\}. $$
Then
\begin{equation}\label{eq:cube-random-intersection-tail}
 \alpha_n(a)^{m-1}\ls\left(\frac C{\sqrt\Lambda}\right)^n.
\end{equation}
\end{lemma}

\begin{proof}
Let $G=(g_1,\ldots,g_n)$ be a standard Gaussian vector and write $G=R\Theta$, where $R=|G|$ and $\Theta$ is uniform on $S^{n-1}$ and independent of $R$.
Put
$$ s=\frac1{2a}=\frac{\sqrt\Lambda}{2A\sqrt n},\qquad x=2s\sqrt n=\frac{\sqrt\Lambda}{A}. $$
On the event $\{R\ls2\sqrt n\}$, the inequality $\|\Theta\|_\infty\ls s$ implies $\max_i|g_i|\ls x$.
Thus, if $q=\Pp\{|g_1|>x\}$, then
\begin{equation}\label{eq:cube-spherical-small-ball-gaussian}
 \alpha_n(a)\ls(1-q)^n+\Pp\{R>2\sqrt n\}\ls e^{-nq}+e^{-c_0n}
\end{equation}
for an absolute constant $c_0>0$.
The standard Gaussian tail estimate gives
\begin{equation}\label{eq:cube-spherical-small-ball-q}
 q\gr\frac{c}{1+x}e^{-x^2/2}\gr\frac{c}{\sqrt\Lambda}\exp\left(-\frac{\Lambda}{2A^2}\right).
\end{equation}
Choose $A$ sufficiently large and put $\eta=1/(2A^2)$.
Since $\Lambda\ls n$, after changing the absolute constants in the bounded range of $n$ we may combine \eqref{eq:cube-spherical-small-ball-gaussian} and \eqref{eq:cube-spherical-small-ball-q} to obtain
$$ \alpha_n(a)\ls e^{-cnq}. $$
For all sufficiently large $\Lambda$, the definition of $\Lambda$ gives $m-1\gr ce^\Lambda$.
Hence
$$ (m-1)nq\gr c n\frac{e^{(1-\eta)\Lambda}}{\sqrt\Lambda}\gr C_1n\ln\Lambda, $$
and the last quantity dominates $n\ln\Lambda$ for all sufficiently large $\Lambda$.
It follows that
$$ \alpha_n(a)^{m-1}\ls\Lambda^{-n}. $$
For bounded $\Lambda$ the assertion follows after increasing $C$.
\end{proof}

\begin{proof}[Proof of Theorem~\ref{th:unconditional-random-intersection-volume}]
We first prove the deterministic lower estimate.
By~\cite[Proposition~2.3]{Giannopoulos-Hartzoulaki-Tsolomitis-2005}, every unconditional isotropic convex body $K$ satisfies $K\supseteq cB_\infty^n$.
Since $Q_n=\frac12B_\infty^n$, it is enough to prove \eqref{eq:unconditional-random-intersection-lower} for $K_1=\cdots=K_m=Q_n$.
Put
$$ P_{m,n}=\conv\{\pm U_ie_j:1\ls i\ls m,\ 1\ls j\ls n\}. $$
Since $Q_n^\circ=2B_1^n$, polarity gives
\begin{equation}\label{eq:cube-random-intersection-polar}
 \left(\bigcap_{i=1}^mU_iQ_n\right)^\circ=2P_{m,n}.
\end{equation}
The body $P_{m,n}$ is contained in $B_2^n$ and has at most $2mn$ vertices.
The classical few-vertex estimates of Carl and Pajor~\cite{Carl-Pajor-1988} and B\'ar\'any and F\"uredi~\cite{Barany-Furedi-1988} imply
$$ \vrad(P_{m,n})\ls C\min\left\{1,\sqrt{\frac{\ln(e+m)}n}\right\}\ls C\sqrt{\frac{\Lambda_{m,n}}n}. $$
By the reverse Santal\'o inequality and \eqref{eq:cube-random-intersection-polar},
$$ \left|\bigcap_{i=1}^mU_iQ_n\right|^{1/n}\gr c\frac{\omega_n^{1/n}}{\vrad(P_{m,n})}\gr\frac c{\sqrt{\Lambda_{m,n}}}. $$
This proves \eqref{eq:unconditional-random-intersection-lower}.

We next prove \eqref{eq:cube-random-intersection-mean-volume}.
Since $\rho_{Q_n}(\theta)=(2\|\theta\|_\infty)^{-1}$, the exact radial-tail identity \eqref{eq:exact-radial-tail-identity} with $p=n$ gives
$$ \E\left|\bigcap_{i=1}^mU_iQ_n\right|=n\omega_n\int_0^\infty t^{n-1}\alpha_n(t)^m\,dt. $$
Let $a$ be as in Lemma~\ref{lem:cube-random-intersection-tail}.
Since $\alpha_n$ is non-increasing and $|Q_n|=1$,
\begin{align*}
 \E\left|\bigcap_{i=1}^mU_iQ_n\right|
 &\ls\omega_na^n+\alpha_n(a)^{m-1}n\omega_n\int_a^\infty t^{n-1}\alpha_n(t)\,dt\\
 &\ls\omega_na^n+\alpha_n(a)^{m-1}\\
 &\ls\left(\frac C{\sqrt{\Lambda_{m,n}}}\right)^n.
\end{align*}
This proves \eqref{eq:cube-random-intersection-mean-volume}.
Finally, Markov's inequality gives
$$ \Pp\left\{\left|\bigcap_{i=1}^mU_iQ_n\right|^{1/n}>\frac{A_0}{\sqrt{\Lambda_{m,n}}}\right\}\ls\left(\frac C{A_0}\right)^n. $$
Choosing $A_0$ sufficiently large and using the deterministic lower estimate proves \eqref{eq:cube-random-intersection-probability}.
\end{proof}

We finish by recording some more general consequences of the same point of view, without pursuing them further here.
For symmetric convex bodies $C_1,\ldots,C_m$ of volume one and every non-empty proper subset $S\subsetneq\{1,\ldots,m\}$, the exact radial-tail identity gives
\begin{equation}\label{eq:volume-upper-subset-prose}
 \E\left|\bigcap_{i=1}^mU_iC_i\right|\ls\min\left\{1,\omega_n\left(\frac{2}{\min_{i\in S}M(C_i)}\right)^n+\exp\left(-\sum_{i\in S}d(C_i)\right)\right\}.
\end{equation}
Indeed, one splits the radial integral at $2/\min_{i\in S}M(C_i)$, uses the definition of $d(C_i)$ on the tail factors corresponding to $i\in S$, and keeps one remaining radial factor in order to integrate the tail.
The Euclidean term in \eqref{eq:volume-upper-subset-prose} is necessary, as is seen from the volume-one Euclidean ball.
More generally, for $\overline B_p^n=|B_p^n|^{-1/n}B_p^n$ and $p>2$ one has $\sqrt n\,M(\overline B_p^n)\simeq\sqrt{\min\{p,\ln(en)\}}$; compare~\cite[Remark~4.7]{Brazitikos-Stavrakakis-2014}.
If $\sqrt n\,M(C_i)\gr B_0$ on a set of $r$ indices, \eqref{eq:volume-upper-subset-prose} gives an exponential estimate involving the sum of the corresponding $d(C_i)$'s, with one index omitted when $r=m$.
Since $d(C)\gr c\sqrt n$ for symmetric isotropic bodies by~\cite[Section~4]{Brazitikos-Stavrakakis-2014}, this yields
$$ \E\left|\bigcap_{i=1}^mU_iC_i\right|\ls C\exp\left[-c\min\left\{n,\min\{r,m-1\}\sqrt n\right\}\right] $$
whenever $r$ of the bodies satisfy the above non-Euclideanity condition.

There is also a general deterministic lower estimate which contains all positive spherical moments simultaneously.
If $K_1,\ldots,K_m$ contain the origin in their interiors, then for every $q>0$ and every $U_1,\ldots,U_m\in O(n)$,
\begin{equation}\label{eq:volume-intersection-all-orders-prose}
 \left|\bigcap_{i=1}^mU_iK_i\right|\gr\omega_n\left(\sum_{i=1}^mM_q(K_i)^q\right)^{-n/q}.
\end{equation}
This follows immediately from $p_{\cap_iU_iK_i}=\max_i p_{K_i}\circ U_i^*$, the power-mean inequality and polar integration.
For isotropic bodies, the standard spherical concentration inequality applied to the gauge gives $M_q(K)\ls M(K)+C\sqrt{q/n}$ for $1\ls q\ls n$.
Combining this with \eqref{eq:isotropic-spherical-means} and choosing $q\simeq\ln(em)$ in \eqref{eq:volume-intersection-all-orders-prose} gives
$$ \left|\bigcap_{i=1}^mU_iK_i\right|\gr\left(\frac c{\sqrt{\min\{n,\ln(enm)\}}}\right)^n. $$
For a fixed number of isotropic bodies one can do better and obtain $|\cap_iU_iK_i|\gr c^{n(m-1)}$.
One way to see this is to represent the volume of the intersection as the value at the origin of a convolution-type log-concave density, use Pr\'ekopa's theorem~\cite{Prekopa-1973} and Fradelizi's barycentric estimate~\cite{Fradelizi-1997}, and then apply the reverse Brunn--Minkowski estimate in isotropic $M$-position~\cite[Theorem~4.5 and Remark~4.6]{Giannopoulos-Paouris-Vritsiou-2014}.
We do not develop these extensions here, since Theorem~\ref{th:unconditional-random-intersection-volume} gives a complete sharp answer in the unconditional class.

\appendix
\section{Calculations for the cube and estimates for \texorpdfstring{$\mathfrak A_{n,d}(q)$}{A(n,d,q)}}\label{appendix:special-calculations}

We collect here the calculations for the cube, including the optimization in Corollary~\ref{cor:projected-cube-profile}, and the comparison of $\mathfrak A_{n,d}(d)$ with $\mathfrak A_{n,d}(d+1)$.

\begin{lemma}\label{lem:l1-spherical-width}
For $1\ls a\ls\sqrt n$,
$$ \omega\{\theta\in S^{n-1}:\|\theta\|_1\ls a\} \simeq a\sqrt{\ln\left(e+\frac n{a^2}\right)}. $$
\end{lemma}

\begin{proof}
Let $m\simeq a^2$.
The set of $m$-sparse unit vectors is contained in $B_2^n\cap aB_1^n$, while the usual block decomposition gives
$$ B_2^n\cap aB_1^n\subseteq2\conv\bigcup_{|I|\ls m}B_2^I. $$
The assertion follows from the sparse Gaussian-width estimate
$$ \E\max_{|I|=m}|P_IG| \simeq\sqrt{m\ln(e+n/m)}; $$
see \cite[Exercise~9.27]{Vershynin-HDP-2}.
\end{proof}

\begin{proof}[Proof of Proposition~$\ref{prop:cube-localized-width}$]
Since $ h_{Q_n}(\theta)=\frac12\|\theta\|_1, $ we have
$$ T_t(Q_n)=\{\theta\in S^{n-1}:\|\theta\|_1\ls2t\sqrt n\}. $$
Apply Lemma~\ref{lem:l1-spherical-width} with $a=2t\sqrt n$.
If $t<1/(2\sqrt n)$, the assertion follows from $\|\theta\|_1\gr1$ on the sphere.
\end{proof}

\begin{proof}[Proof of Proposition~$\ref{prop:cube-dominated-projection-profile}$]
Since $h_K(\theta)\gr a\|\theta\|_1$, for every $t>0$,
$$ T_t(K) \subseteq\left\{\theta\in S^{n-1}: \|\theta\|_1\ls\frac{t\sqrt n}{a}\right\}. $$
Put
$$ L_d=\ln\frac{en}{d}, \qquad t=c_0a\sqrt{\frac d{nL_d}}, \qquad A=\frac{t\sqrt n}{a} =c_0\sqrt{\frac d{L_d}}. $$
If $A<1$, then $T_t(K)$ is empty.
Otherwise, Lemma~\ref{lem:l1-spherical-width} gives
$$ \omega_K(t) \ls CA\sqrt{\ln\left(e+\frac n{A^2}\right)} \ls Cc_0\sqrt{\frac d{L_d}} \sqrt{\ln\left(e+\frac{nL_d}{c_0^2d}\right)} \ls C'c_0\sqrt d. $$
Choose $c_0>0$ sufficiently small and apply Proposition~\ref{prop:localized-width-reduction}.
This gives
$$ r(P_FK)\gr ca\sqrt{\frac d{\ln(en/d)}} $$
with the stated probability.
Finally, $ K\supseteq aB_\infty^n\supseteq aB_2^n, $ so $r(P_FK)\gr a$ deterministically.
Combining the two estimates proves \eqref{eq:cube-dominated-projection-profile}.
\end{proof}

\begin{proof}[Proof of Proposition~$\ref{prop:cube-gelfand-obstruction}$]
Since $Q_n^\circ=2B_1^n$,
$$ R(Q_n^\circ\cap F) =2\sup\{|x|:x\in F,\ \|x\|_1\ls1\}. $$
The right-hand side is bounded below by twice the $d$-th Gelfand width of $B_1^n$ in $\ell_2^n$.
By~\cite[Theorem~1.1]{Foucart-Pajor-Rauhut-Ullrich-2010},
$$ c_d(B_1^n,\ell_2^n) \simeq\min\left\{1, \sqrt{\frac{1+\ln(n/d)}d}\right\}, $$
which is equivalent to the right-hand side of \eqref{eq:cube-gelfand-obstruction}.
Since the estimate holds for every $F$, it gives \eqref{eq:weak-profile-cube-lower} for every $q>0$.
\end{proof}

We next compare the two quantities $\mathfrak A_{n,d}(d)$ and $\mathfrak A_{n,d}(d+1)$.

\begin{theorem}\label{th:adjacent-weak-projection-scales}
For every $1\ls d<n$,
\begin{equation}\label{eq:adjacent-projection-exponents}
 \mathfrak A_{n,d}(d) \ls\mathfrak A_{n,d}(d+1) \ls C\mathfrak A_{n,d}(d).
\end{equation}
If $\mathfrak A_{n,d}=\mathfrak A_{n,d}(d)$, then
\begin{equation}\label{eq:weak-projection-profile-bounds}
 c\sqrt n\min\left\{1, \sqrt{\frac{\ln(en/d)}d}\right\} \ls\mathfrak A_{n,d} \ls C\min\left\{\sqrt n, \exp\left(\frac{C(n-d)}{d+1}\right)\right\}.
\end{equation}
\end{theorem}

\begin{proof}
Monotonicity gives the first inequality in \eqref{eq:adjacent-projection-exponents}.
Moreover, $r(K)\gr cL_K\gr c$, so
$$ Y_K(F)=\sqrt n\,R(K^\circ\cap F)\ls C\sqrt n $$
for every $F\in G_{n,n-d}$.
If a non-negative random variable $Y\ls M$ satisfies $\Pp\{Y>A\tau\}\ls\tau^{-d}$, then
$$ \|Y\|_{L_{d+1,\infty}}\ls A^{d/(d+1)}M^{1/(d+1)}. $$
Applying this with $Y=Y_K$ and using Proposition~\ref{prop:cube-gelfand-obstruction}, which gives $\mathfrak A_{n,d}(d)\gr c\sqrt{n/d}$, proves the second inequality in \eqref{eq:adjacent-projection-exponents}.
The lower bound in \eqref{eq:weak-projection-profile-bounds} is Proposition~\ref{prop:cube-gelfand-obstruction}; the upper bound follows from Theorem~\ref{th:projection-tail} and the deterministic estimate $Y_K\ls C\sqrt n$.
\end{proof}

Consequently, $\mathfrak A_{n,d}\simeq\sqrt n$ when $d\ls c\ln(en/d)$, while $\mathfrak A_{n,d}\simeq_\delta1$ when $d\gr\delta n$.

\begin{proof}[Proof of Proposition~$\ref{prop:cube-negative-moment-profile}$]
Let $G=(g_1,\ldots,g_n)$ be a standard Gaussian vector and write $ G=R\Theta$, $R=|G|, $ where $\Theta$ is uniform on $S^{n-1}$ and independent of $R$.
Put $ Z=\|G\|_\infty. $
We first assume $0<p\ls n/2$.
Since $Z=R\|\Theta\|_\infty$,
\begin{equation}\label{eq:cube-gaussian-factorization}
 M_{-p}(B_\infty^n) =(\E R^{-p})^{1/p}(\E Z^{-p})^{-1/p}.
\end{equation}
Since $R^2$ has the chi-square distribution with $n$ degrees of freedom,
$$ \E R^{-p}=2^{-p/2}\frac{\Gamma((n-p)/2)}{\Gamma(n/2)}. $$
The standard gamma-ratio estimates give
\begin{equation}\label{eq:negative-chi-moment}
 (\E R^{-p})^{1/p}\simeq\frac1{\sqrt n}, \qquad 0<p\ls n/2.
\end{equation}

We claim that
\begin{equation}\label{eq:gaussian-maximum-negative-moment}
 (\E Z^{-p})^{-1/p} \simeq\sqrt{\ln\left(e+\frac n{1+p}\right)}, \qquad 0<p\ls n/2.
\end{equation}
Let $ H(t)=\Pp\{|g_1|\ls t\}. $
For $1\ls p\ls n/2$ define $t_p$ by $ H(t_p)=e^{-p/n}. $
The two-sided Gaussian tail estimates imply
$$ t_p\simeq\sqrt{\ln\left(e+\frac np\right)}. $$
Since $\Pp\{Z\ls t_p\}=e^{-p}$ then $ \E Z^{-p}\gr e^{-p}t_p^{-p}, $ which gives $ (\E Z^{-p})^{-1/p}\ls et_p. $

For the reverse inequality, the variables $H(|g_i|)$ are independent and uniform on $[0,1]$.
Hence $ S=-n\ln H(Z) $ has the standard exponential distribution.
If $ f_n(s)=H^{-1}(e^{-s/n}), $ then
\begin{equation}\label{eq:maximum-exponential-representation}
 \E Z^{-p}=\int_0^\infty f_n(s)^{-p}e^{-s}\,ds.
\end{equation}
Uniform Gaussian estimates give
\begin{equation}\label{eq:fn-two-ranges}
 f_n(s)\simeq\sqrt{\ln\frac{en}{s}} \quad(0<s\ls n/2), \qquad f_n(s)\simeq e^{-s/n} \quad(s\gr n/2).
\end{equation}
Put $L=\ln(en/p)$ and $s_0=pe^{L/2}=\sqrt{enp}$.
On $0<s\ls p$ we have $f_n(s)\gr f_n(p)=t_p$.
On $ p\ls s\ls\min\{s_0,n/2\}, $ we have $f_n(s)\gr c\sqrt L\gr ct_p$.
If $s_0<n/2$, then the contribution of $[s_0,n/2]$ is at most $C^pe^{-s_0}$, which is bounded by $(C/t_p)^p$ because $s_0\gr cpL$.
Finally, by the second estimate in \eqref{eq:fn-two-ranges} and $p\ls n/2$,
$$ \int_{n/2}^\infty f_n(s)^{-p}e^{-s}\,ds \ls C^p\int_{n/2}^\infty e^{ps/n-s}\,ds \ls2C^pe^{-n/4} \ls(C/t_p)^p. $$
Together with \eqref{eq:maximum-exponential-representation}, this proves $ \E Z^{-p}\ls(C/t_p)^p. $
Thus \eqref{eq:gaussian-maximum-negative-moment} holds for $1\ls p\ls n/2$.
If $0<p<1$, monotonicity of power means gives $ (\E Z^{-1})^{-1}\ls(\E Z^{-p})^{-1/p}\ls\E Z, $ and both endpoints are comparable to $\sqrt{\ln(en)}$.

Combining \eqref{eq:cube-gaussian-factorization}, \eqref{eq:negative-chi-moment} and \eqref{eq:gaussian-maximum-negative-moment} proves \eqref{eq:cube-negative-moment-profile} for $p\ls n/2$.
If $p>n/2$, monotonicity and $ \min_{S^{n-1}}\|\theta\|_\infty=1/{\sqrt n} $ give
$$ \frac1{\sqrt n}\ls M_{-p}(B_\infty^n) \ls M_{-n/2}(B_\infty^n)\ls\frac C{\sqrt n}. $$
Since $\ln(e+n/(1+p))\simeq1$ in this range, the proof is complete.
\end{proof}

\begin{proof}[Proof of Corollary~$\ref{cor:projected-cube-profile}$]
Apply Theorem~\ref{th:projected-weighted-negative-profile} with $K_j=Q_n$, $\lambda_j=1/m$, $q=mp$ and $u=k$.
Since $p_{Q_n}=2\|\cdot\|_\infty$, Proposition~\ref{prop:cube-negative-moment-profile} gives
$$ a_q(\lambda)=M_{-p}(Q_n)\simeq\sqrt{\frac{\ell_p}{n}}, \qquad M(Q_n)\simeq\sqrt{\frac{\ln(en)}n}, \qquad b(Q_n)=2. $$
Moreover, $ \overline M_\lambda=M(Q_n)$ and $B_\lambda=2/{\sqrt m}. $
Combining the two inclusions in Theorem~\ref{th:projected-weighted-negative-profile} proves \eqref{eq:projected-cube-profile-general}.

For the second assertion choose $ p=({Ak\ln(en)})/m, $ where $A>0$ is a sufficiently large absolute constant.
Then $mp\gr2k$ and $\Delta_p\gr ck\ln(en)$.
Moreover,
$$ 1+\frac{Ak\ln(en)}m\simeq\max\left\{1,\frac{k\ln(en)}m\right\}, $$
and hence $\ell_p\simeq L_{m,k}$.
Since
$$ \frac{k-1}{\Delta_p}\ln\left(C\sqrt{\frac n{\ell_p}}\right)\ls C, \qquad \frac{k}{\Delta_p}\ls\frac C{\ln(en)}, $$
the last two factors in \eqref{eq:projected-cube-profile-general} are bounded by an absolute constant.
This proves \eqref{eq:projected-cube-profile-optimized}.
\end{proof}

\bigskip

\noindent {\bf Acknowledgements.} The first named author acknowledges support from a PhD scholarship of the National Technical University of Athens.
The second named author acknowledges support from the Hellenic Foundation for Research and Innovation (H.F.R.I.) under the ``4th Call for H.F.R.I. research projects to support Postdoctoral Researchers'' (Project Number: 28948).
The authors thank Apostolos Giannopoulos for useful discussions.

\bigskip

\noindent {\bf Keywords:} 
isotropic convex bodies, support functions, negative moments, random rotations, random projections, Dvoretzky theorem.

\smallskip

\noindent {\bf 2020 Mathematics Subject Classification:}
Primary 52A23; Secondary 46B06, 46B07, 52A22, 60D05.

\bigskip

\noindent \textsc{Antonios Hmadi}: School of Applied Mathematical and Physical Sciences, National Technical University of Athens, Department of Mathematics, Zografou Campus, GR-157 80, Athens, Greece.

\smallskip

\noindent \textit{E-mail:} \texttt{ahmadi@mail.ntua.gr}

\bigskip

\noindent \textsc{Dimitrios-Marios Liakopoulos}: School of Applied Mathematical and Physical Sciences, National Technical University of Athens, Department of Mathematics, Zografou Campus, GR-157 80, Athens, Greece.

\smallskip

\noindent \textit{E-mail:} \texttt{dm\_liakopoulos@mail.ntua.gr}

\end{document}